\documentclass[10pt]{amsart}

\usepackage{amsmath, amscd, amssymb, euler}
\usepackage[frame,cmtip,arrow,matrix,line,graph,curve]{xy}
\usepackage{graphpap, color}
\usepackage[mathscr]{eucal}

\numberwithin{equation}{section}

\newcommand{\CC}{\mathbb{C}}

\newcommand{\PP}{\mathbb{P}}

\newcommand{\ZZ}{\mathbb{Z}}

\newcommand{\bP}{\mathbf{P}}
\newcommand{\bX}{\mathbf{X}}
\newcommand{\bS}{\mathbf{S}}
\newcommand{\bM}{\mathbf{M}}
\newcommand{\bH}{\mathbf{H}}
\newcommand{\bC}{\mathbf{C}}

\newcommand{\cal}{\mathcal}

\def\cC{{\cal C}}

\def\cE{{\cal E}}
\def\cF{{\cal F}}
\def\cH{{\cal H}}

\def\cM{{\cal M}}

\def\cO{{\cal O}}

\def\cS{{\cal S}}

\def\cU{{\cal U}}

\def\cZ{{\cal Z}}
\def\cI{{\cal I}}

\def\mapright#1{\,\smash{\mathop{\lra}\limits^{#1}}\,}

\def\lra{\longrightarrow}

\def\begeq{\begin{equation}}
\def\endeq{\end{equation}}
\def\and{\quad{\rm and}\quad}

\def\and{\quad\text{and}\quad}
\def\mapright#1{\,\smash{\mathop{\lra}\limits^{#1}}\,}

\DeclareMathOperator{\Ext}{Ext} 

\DeclareMathOperator{\Hom}{Hom} 
 
 \DeclareMathOperator{\id}{id}

\DeclareMathOperator{\Sym}{Sym} 

\newtheorem{prop}{Proposition}[section]
\newtheorem{theo}[prop]{Theorem}
\newtheorem{lemm}[prop]{Lemma}
\newtheorem{coro}[prop]{Corollary}
\newtheorem{rema}[prop]{Remark}
\newtheorem{exam}[prop]{Example}
\newtheorem{defi}[prop]{Definition}

\let\lab=\label

\def\ev{\mathrm{ev}}

\def\PP{\mathbb{P}}
\def\CC{\mathbb{C}}

\def\lra{\longrightarrow}
\def\mapright#1{\,\smash{\mathop{\lra}\limits^{#1}}\,}
\def\cO{\mathcal{O}}

\def\Spec{\mathrm{Spec}}
\def\git{/\!/ }

\title{Hilbert scheme of rational cubic curves via stable maps}

\author{Kiryong Chung}
\address{Department of Mathematics, Seoul National University, Seoul 151-747, Korea}
\email{dragon10@snu.ac.kr}

\author{Young-Hoon Kiem}
\address{Department of Mathematics and Research Institute
of Mathematics, Seoul National University, Seoul 151-747, Korea}
\email{kiem@math.snu.ac.kr}

\date{}

\thanks{Partially supported by KOSEF grant
R01-2007-000-20064-0.}

\begin{document}
\begin{abstract}
The space of smooth rational cubic curves in projective space
$\PP^r$ ($r\ge 3$) is a smooth quasi-projective variety, which gives
us an open subset of the corresponding Hilbert scheme, the moduli
space of stable maps, or the moduli space of stable sheaves. By
taking its closure, we obtain three compactifications $\bH$, $\bM$,
and $\bS$ respectively. In this paper, we compare these
compactifications. First, we prove that $\bH$ is the blow-up of
$\bS$ along a smooth subvariety which is the locus of stable sheaves
which are planar (i.e. support is contained in a plane). Next we
prove that $\bS$ is obtained from $\bM$ by three blow-ups followed
by three blow-downs and the centers are described explicitly. Using
this, we calculate the cohomology of $\bS$.
\end{abstract}
\maketitle

\section{Introduction}

Let $\bX_0$ be the space of smooth rational curves of
degree $d$ in $\PP^r$. It is easy to see that $\bX_0$ is a smooth
quasi-projective variety of dimension $(r+1)(d+1)-4$. From moduli
theoretic point of view, the following questions are quite natural.
\begin{enumerate}
\item Does $\bX_0$ admit a moduli theoretic compactification?
\item If there are more than one such compactifications,
what are the relationships among them?
\item Can we calculate the differences of intersection numbers of
cycles coming from geometric conditions?\end{enumerate}

For the first question, there are several well-known
compactifications as we will review below. The purpose of this
paper is to provide answers to the second question for $d=3$. Note
that the second question is trivial when $d=1$ because $X_0$ is
compact. The case where $d=2$ has been worked out in \cite{kiem}:
The moduli space of stable maps $\cM_0(\PP^r,2)$ is Kirwan's
partial desingularization of the quasi-map space $\PP
(\Sym^2(\CC^2)\otimes \CC^{r+1})/\!/ SL(2)$ and the Hilbert scheme
$\cH ilb^{2m+1}(\PP^r)$ is obtained from $\cM_0(\PP^r,2)$ by a
blow-up followed by a blow-down. The third question is related to
the problem of comparing various curve counting invariants, such
as the Gromov-Witten invariant, Donaldson-Thomas invariant and
Pandharipande-Thomas invariant. A successful comparison of these
curve counting invariants may be achieved if the second question
is answered in a satisfactory fashion. For instance, if we can
describe the birational maps between two different
compactifications of $\bX_0$ in terms of blow-ups and -downs whose
centers are themselves moduli spaces for lower degree curves, then
it is quite plausible that the differences of the curve counting
invariants may be expressed as inductive formulae. In this paper
we first review several natural moduli theoretic compactifications
and then compare these compactifications in terms of explicit
blow-ups and -downs.

In \S2, we review several natural moduli theoretic compacticiations
of $\bX_0$. The first compactification comes from geometric
invariant theory (GIT). A smooth rational cubic curve is given by a
map $f:\PP^1\to \PP^r$. If we choose homogeneous coordinates of
$\PP^1$ and $\PP^r$, $f$ is given by an $(r+1)$-tuple of homogeneous
polynomials of degree $d$ in two variables $z_0,z_1$. To remove the
dependency on the choice of homogeneous coordinates, we have to take
the quotient by the action of $Aut(\PP^1)=PGL(2)$. Hence, we obtain
a compactification by GIT quotient, often called the \emph{quasi-map
space}
\[
\bX=\PP(\Sym^3(\CC^2)\otimes \CC^{r+1})\git SL(2).
\]
The strength of this
compactification $\bX$ is that it is easy to calculate the
cohomology ring or Chow ring or K-groups by using the equivariant
Morse theory, or the Atiyah-Bott-Kirwan theory \cite{K2}. The
weakness of $\bX$ is that the boundary points do not have natural
geometric meaning.

The second compactification is obtained from the Hilbert scheme.
We have the obvious embedding $\bX_0\hookrightarrow \cH
ilb^{3m+1}(\PP^r)$ of $\bX_0$ into the Hilbert scheme of closed
subschemes with Hilbert polynomial $3m+1$. This turns out to be an
open immersion and by taking its closure we obtain a
compactification $\bH$ of $\bX_0$, which we call the \emph{Hilbert
compactificaiton}.

The third compactification comes from Kontsevich's moduli space of
stable maps. A stable map is a morphism of a connected nodal curve
$f:C\to \PP^r$ with finite automorphism group. Let
$\cM_0(\PP^r,d)$ denote the moduli space of stable maps of
arithmetic genus $0$ and degree $d$. It is well-known that this is
an irreducible normal projective scheme. The obvious inclusion
$\bX_0\hookrightarrow \cM_0(\PP^r,3)$ is an open immersion and
$\bM=\cM_0(\PP^r,3)$ is a compactification of $\bX_0$, which we
call the \emph{Kontsevich compactification}. In \cite{KM}, we
proved that $\bM$ is obtained from $\bX$ by three blow-ups and two
blow-downs and the blow-up/-down centers are explicitly described
in terms of moduli spaces of stable maps of degrees 1 and 2.
\begin{theo}\lab{thm1.2} \cite{KM}
The birational map $\bX\dashrightarrow \bM$ is the composition of
three blow-ups followed by two blow-downs. The blow-up centers are
respectively, $\PP^{r}$, $\cM_{0,2}(\PP^{r},1)/S_2$ (where $S_2$
interchanges the two marked points) and the blow-up of
$\cM_{0,1}(\PP^{r},2)$ along the locus of three irreducible
components. The centers of the blow-downs are respectively the
$S_2$-quotient of a $(\PP^{r})^2$-bundle on $\cM_{0,2}(\PP^{r},1)$
and a $(\PP^{r-1})^3/S_3$-bundle on $\PP^{r}$. Here
$\cM_{0,k}(\PP^r,d)$ denotes the moduli space of stable maps of
genus 0 and degree $d$ with $k$ marked points.
\end{theo}
The following diagram explains how we get $\bM$ from $\bX$ by
explicit blow-ups and -downs.
\begin{equation}\lab{diagramKM}
\xymatrix{ \bP_3\ar[r]^{\pi_3}\ar[d]_{\pi_4}\ar[dr]^{p_3}&
\bP_2\ar[r]^{\pi_2}&\bP_1\ar[r]^{\pi_1}&\bP_0\ar[d]\\
\bP_4\ar[d]_{\pi_5} &\bP_3/SL(2)\ar[d]\ar[rr]\ar[drr]^{\bar\psi_3}
&&\bP_0/SL(2)=\bX\ar@{.>}[d]^{\bar\psi_0}\\
\bP_5\ar[r]^{p_5}& \bP_5/SL(2)\ar[rr]^{\cong}_{\bar\psi_5} &&
\cM_0(\PP^r,3)=\bM }
\end{equation}
Here $\bP_0$ denotes the stable part of $\PP(\Sym^3(\CC^2)\otimes
\CC^{r+1})$. Using this theorem, we could calculate the cohomology
ring and the Picard group of the Kontsevich's moduli space of stable
maps $\cM_0(\PP^r,3)$.

\begin{rema}\emph{ Theorem \ref{thm1.2} holds true for any $r\ge 1$. When $r=1$,
$\pi_3$ is the identity map and $\pi_2$ cancels out $\pi_4$ while
$\pi_1$ cancels out $\pi_5$. Therefore, we have an isomorphism
$$\cM_0(\PP^1,3)\cong \PP(\Sym^3(\CC^2)\otimes \CC^{2})\git
SL(2).$$}
\end{rema}

The fourth compactification is by Simpson's moduli space of stable
sheaves. Recall that a coherent sheaf $E$ is \emph{pure} if any
nonzero subsheaf of $E$ has the same dimensional support as $E$. A
pure sheaf $E$ is called \emph{semistable} if \[
\frac{\chi(E(m))}{r(E)}\le \frac{\chi(E'(m))}{r(E')}\qquad \text{for
}m>>0
\]
for any nontrivial pure quotient sheaf $E'$ of the same dimension,
where $r(E)$ denotes the leading coefficient of the Hilbert
polynomial $\chi(E(m))$. We obtain \emph{stability} if $\le $ is
replaced by $<$. Simpson proved that there is a projective moduli
scheme $\cS imp^{P}(\PP^r)$ of semistable sheaves of given Hilbert
polynomial $P$. It is easy to see that semistability coincides with
stability when $P(m)=3m+1$. If $C$ is a smooth rational cubic curve
in $\PP^r$, then the structure sheaf $\cO_C$ is a stable sheaf.
Hence we get an open immersion $\bX_0\hookrightarrow \cS
imp^{3m+1}(\PP^r)$. By taking the closure we obtain a
compactifiction $\bS$, which we call the \emph{Simpson
compactification}.

In \S3, we compare the Hilbert compactification $\bH$ with the
Simpson compactification $\bS$. For $r=3$, Freiermuth and Trautmann
proved in \cite{FT} that $\bH\cong \bS$. For arbitrary $r$, we prove
the following
\begin{prop}\lab{prop1.4}
There is a morphism $\bH\to \bS$ which is the blow-up along the
smooth locus of stable sheaves with planar support.
\end{prop} When $r=3$, the locus of planar stable sheaves is a
divisor and hence we obtain an isomorphism $\bH\cong \bS$. One
direct way to prove this proposition is to construct a family of
stable sheaves parameterized by $\bH$. The structure sheaves of
the closed subschemes parameterized by $\bH$ are stable except
along the locus of planar cubic curves, which is a divisor. By
applying elementary modification, we obtain a family of stable
sheaves and thus a morphism from $\bH$ to $\bS$. Then one can
check that this is a blow-down. Another way to prove this is to
use a result of Freiermuth and Trautmann (\cite{FT}) which we
explain in \S3.

In \S4, we compare the Kontsevich compactification $\bM$ and the
Simpson compactification $\bS$. Let $f:C\to \PP^r$ be a stable
map. Then $f_*\cO_C$ is a coherent sheaf on $\PP^r$. The locus of
unstable sheaves turns out to consist of two irreducible
components $\Gamma^1\cup \Gamma^2$, where $\Gamma^1$ is the locus
of stable maps with linear image (i.e. the image is a line) while
$\Gamma^2$ is the locus of stable maps with bilinear image (i.e.
the image is the union of two lines). To resolve the indeterminacy
we first blow up along $\Gamma^1$ and apply elementary
modification with respect to the destabilizing subsheaves which we
define as the first nonzero terms in the Harder-Narasimhan
filtrations. Then the locus of unstable sheaves has still two
components: one is the proper transform of $\Gamma^2$ and the
other $\Gamma^3$ is a subvariety of the exceptional divisor of the
blow-up. We then blow up along $\Gamma^2$ and apply the elementary
modification with respect to the destabilizing subsheaves. Again
we blow up along $\Gamma^3$ and apply elementary modification.
After these three blow-ups, we obtain a family of stable sheaves
and hence a morphism to $\bS$. Then we study the geometry of the
exceptional divisors. It turns out that the exceptional divisor of
the second blow-up becomes a weighted projective bundle over a
variety and we can contract this divisor. Then the exceptional
divisor of the third blow-up becomes a weighted projective bundle
and we can contract this divisor. Finally we contract the
exceptional divisor of the first blow-up in a similar fashion. In
this local analysis, the main technique is the variation of GIT
quotients \cite{Thad, DH}. Then it is easy to check that the
morphism to $\bS$ factors through the blow-downs and the induced
map is bijective. By the generalized Riemann existence theorem
\cite[p442]{Hartshorne}, we deduce that all morphisms are
algebraic and the blown-down spaces are algebraic. So we obtain
the following
\begin{theo}\lab{mainthm1}
$\bS$ is obtained from $\bM$ by blowing up along $\Gamma^1$,
$\Gamma^2_1$, $\Gamma^3_2$ and then blowing down along $\Gamma^2_3$,
$\Gamma^3_4$, $\Gamma^1_5$, where $\Gamma^i_j$ denote the proper
transform of $\Gamma^i$ at the $j$th stage.
\end{theo}
The following diagram summarizes the results of this paper and
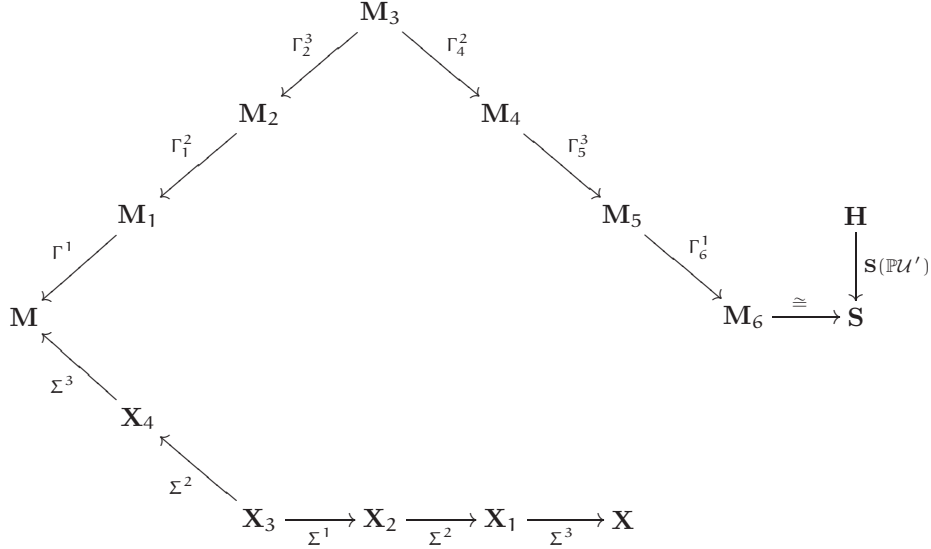
\cite{KM}.
\[
\xymatrix{ &&&\bM_3\ar[dl]_{\Gamma_2^3}\ar[dr]^{\Gamma_4^2}\\
&&\bM_2\ar[dl]_{\Gamma_1^2}&&\bM_4\ar[dr]^{\Gamma^3_5}\\
&\bM_1\ar[dl]_{\Gamma^1}&&&&\bM_5\ar[dr]^{\Gamma^1_6}&&\bH\ar[d]^{\bS(\PP \cU')}\\
\bM&&&&&&\bM_6\ar[r]^\cong & \bS \\
&\bX_4\ar[ul]^{\Sigma^3}\\
&&\bX_3\ar[ul]^{\Sigma^2}\ar[r]_{\Sigma^1}&\bX_2\ar[r]_{\Sigma^2}&\bX_1\ar[r]_{\Sigma^3}&\bX
}\] All the arrows are blow-ups and the blow-up centers are
indicated above the arrows.

In \S5, we calculate the Betti numbers of $\bS$ by using Theorem
\ref{mainthm1}. When $r=3$, we get exactly the same numbers,
calculated by Ellingsrud, Piene and Stromme \cite{EPS}.

There are other interesting compactifications, such as the Chow
compactification, the variety of nets of quadrics \cite{EPS}, the
Vainsencher-Xavier compactification \cite{VX} and the variety of
triples \cite{Pi}. When $r=3$, the variety of nets of quadrics was
shown to be a blow-down of $\bH$ but the relationships for other
compactifications are not known. We hope to compare them with $\bX,
\bH, \bM, \bS$ in the future.

\section{Compactifications of the space of rational cubics}
In this section, we recall several well-known compactifications of
the space of curves: compactifications by the Hilbert scheme,
Kontsevich's moduli space of stable maps, Simpson's moduli space of
stable sheaves and the space of quasi-maps. Our goal is to compare
these compactifications. We fix a positive integer $r$.

\subsection{Compactification by quasi-maps}
This is perhaps the easiest to describe. A smooth rational curve of
degree 3 in projective space $\PP^r$, is given by an $(r+1)$-tuple
$(f_0:f_1:\cdots:f_r)$ of degree 3 homogeneous polynomials in two
variables $z_0,z_1$, the homogeneous coordinates of $\PP^1$. Upon
fixing the basis $z_0^3,z_0^2z_1,z_0z_1^2,z_1^3$ of the space of
degree 3 homogeneous polynomials, the curve is determined by a
$4\times (r+1)$ matrix of coefficients. Whenever this matrix has
maximal rank, we get a smooth rational cubic in $\PP^r$ and two such
matrices determine the same curve if and only if they are in the
same orbit under the action of $Aut(\PP^1)=PGL(2)$. Hence the space
of rational cubics in $\PP^r$ can be described as
\[
\bX_0:=\PP(\Sym ^3(\CC^2)\otimes \CC^{r+1})_4/PGL(2)
\]
where the subscript $4$ denotes the open subset of rank 4 matrices.
Thus, Geometric Invariant Theory (GIT) provides us with a natural
compactification
\[
\bX:=\PP(\Sym ^3(\CC^2)\otimes \CC^{r+1})/\!/PGL(2)
\]
which is often called, the \emph{quasi-map space}. Thanks to the
Atiyah-Bott-Kirwan theory, we can easily calculate the cohomology
ring/topological K-group/Chow ring of this compactification
\cite{K2,KM}. However, the geometric meaning of its boundary points
is not clear.

\subsection{Kontsevich's moduli space of stable maps}
In early 1990's, Kontsevich introduced the notion of stable maps and
a compactification of the space of smooth curves, called the
\emph{moduli space of stable maps}. See \cite{FP} for an
introduction.

By definition, a \emph{stable map} to $\PP^r$ is a morphism $f:C\to
\PP^r$ of a connected reduced curve $C$ which may have only nodal
singularities ($xy=0$), such that the automorphism group of $f$ is
finite. If we fix the arithmetic genus $g$ of $C$ and the homology
class $d=f_*[C]$, then we obtain a projective moduli space
$\cM_{g}(\PP^r,d)$, which parameterizes isomorphism classes of
stable maps. When $g=0$, $\cM_0(\PP^r,d)$ is an irreducible normal
projective variety with only finite quotient singularities.

Our concern in this paper is $$\bM:=\cM_0(\PP^r,3)$$ which is
obviously a compactification of the space $\bX_0$ of smooth
rational cubics. As the first step of our project of comparing
various compactifications of $\bX_0$, we proved in \cite{KM}, that
$\bM$ is obtained from $\bX$ by three blow-ups followed by two
blow-downs. The blow-up/-down centers were all described in terms
of moduli spaces of stable maps of degrees 1 and 2. These results
enabled us to do various cohomological calculations on $\bM$. In
this paper, we compare $\bM$ with other moduli theoretic
compactifications described below. We will call $\bM$ the
\emph{Kontsevich compactification}. We will write $\bM(\PP^r)$ if
it is necessary to emphasize the target space $\PP^r$.

\subsection{Hilbert scheme and Chow scheme}
Classical approaches to compactification of the space of smooth
curves are to use either the Hilbert scheme or the Chow scheme.
See \cite{Kollar, Harris} for an introduction.

The Hilbert scheme $\cH ilb^{P}(\PP^r)$ is the projective moduli
space of closed subschemes in $\PP^r$ whose Hilbert polynomial is
$P$. The space of smooth rational cubics in $\PP^r$ form an open
subset of $\cH ilb^{3m+1}(\PP^r)$ and we call its closure $\bH$, the
\emph{Hilbert compactification}. When it is necessary to emphasize
the target space $\PP^r$, we will write $\bH(\PP^r)$ instead of
$\bH$.

The Chow scheme $\cC how^{1,3}(\PP^r)$ of one dimensional cycles of
degree three is a projective scheme which contains $\bX_0$ as an
open subset. We denote by $\bC$ the closure of $\bX_0$ in the Chow
scheme and call it the \emph{Chow compacticiation}. $\bH$ is smooth
and we have a natural morphism $HC:\bH\to \bC$ forgetting the
thickening structure of multiple components. Furthermore, there is a
natural morphism $\bM\to \bC$ which forgets the ramification points
of multiple components.

\subsection{Simpson's moduli space of stable sheaves}
Let $E$ be a coherent sheaf on $\PP^r$. We say $E$ is \emph{pure} if
any nonzero subsheaf of $E$ has the same dimensional support as $E$.
A pure sheaf $E$ is called \emph{semistable} if \[
\frac{\chi(E(m))}{r(E)}\le \frac{\chi(E'(m))}{r(E')}\qquad \text{for
}m>>0
\]
for any nontrivial pure quotient sheaf $E'$ of the same dimension,
where $r(E)$ denotes the leading coefficient of the Hilbert
polynomial $\chi(E(m))$. We obtain \emph{stability} if $\le $ is
replaced by $<$.

In \cite{Simpson}, Simpson proved by GIT that there is a projective
moduli scheme $\cS imp^P(\PP^r)$ of equivalence classes of
semistable sheaves on $\PP^r$ whose Hilbert polynomial is $P$. The
following easy lemma explains why $\cS imp^{3m+1}(\PP^r)$ gives us a
compactification of $\bX_0$.
\begin{lemm}\cite[Lemma 1]{FT}\lab{lem1} If $C$ is a Cohen-Macaulay curve of
degree 3 in $\PP^r$, its structure sheaf $\cO_C$ is a stable sheaf.
\end{lemm}
\begin{proof}
Since $\cO_C$ is a quotient of $\cO_{\PP^r}$, any nontrivial pure
quotient of $\cO_{C}$ is the structure sheaf $\cO_{C'}$ of a closed
subcurve $C'$ of $C$. As the degree of $C'$ is at most two, the
support of $C'$ is a line or a conic or a pair of lines. In any
case,
\[
m+\frac13 =\frac{\chi(\cO_C(m))}{r(\cO_C)} <
\frac{\chi(\cO_{C'}(m))}{r(\cO_{C'})}=m+1, \text{ or } m+\frac12 .
\]
\end{proof}

By an easy calculation, we have an identification of tangent
spaces
$$T_{\cO_C} \cS imp^{3m+1}(\PP^r)=Ext^1_{\PP^r}(\cO_C,\cO_C)\cong
H^0(C,N_{C/\PP^r})=T_{[C]}\bX_0$$ for a smooth rational cubic $C\in
\bX_0.$ Hence, via the map $C\mapsto \cO_C$, $\bX_0$ is isomorphic
to an open subset of $\cS imp^{3m+1}(\PP^r)$. Let $\bS$ be the
irreducible component of $\cS imp^{3m+1}(\PP^r)$ which contains
$\bX_0$. We call this the \emph{Simpson compactification}. When it
is necessary to emphasize the target space $\PP^r$, we will write
$\bS(\PP^r)$ instead of $\bS$.

In \S3, we compare $\bH$ and $\bS$ and in \S4, we compare $\bM$ and
$\bS$. Their relationships will be described in terms of explicit
blow-ups and -down. 

\section{From Hilbert to Simpson}

In this section, we prove that the Hilbert compactification $\bH$
is the blow-up of the Simpson compactification $\bS$ along a
smooth subvariety. One way to prove this is as follows: the
universal family $\cZ\subset \PP^r\times\bH$ defines a family of
sheaves $\cO_\cZ$ on $\PP^r\times\bH$ which is flat over $\bH$.
The locus of unstable sheaves is a smooth divisor $\Delta$ and the
destabilizing subsheaves are zero dimensional. We then apply the
elementary modification
\[
\cF:=\ker ( \cO_{\cZ}\to \cO_{\cZ}|_\Delta\to A)
\]
where $A$ is the destabilizing quotient. Then one can check that
$\cF$ is a flat family of stable sheaves on $\PP^r\times\bH$ and
hence we obtain a morphism $\bH\to \bS$. By further analyzing
fibers, one can show that this is a blow-up map or a divisorial
contraction of $\Delta$ as we do below.

Instead of the above method of using elementary modification, we
take a shorter path of using the results of \cite{PS,FT}.
\begin{theo} \cite{PS}\lab{PS1}
$\cH ilb^{3m+1}(\PP^3)$ has only two irreducible components. They
are smooth and intersect transversely. One of them is $\bH(\PP^3)$
and the other is a 15 dimensional variety parameterizing planar
cubics coupled with points. Their intersection is a divisor
$\Delta(\PP^3)$ of $\bH(\PP^3)$.
\end{theo}

\begin{theo}\cite{FT}\lab{FT1} \begin{enumerate}
\item $\cS imp^{3m+1}(\PP^3)$ is the fine moduli space of stable
sheaves, i.e. semistable sheaves are stable.
\item $\cS imp^{3m+1}(\PP^3)$ has two irreducible components which
intersect transversely along $\Delta(\PP^3)$. One is $\bS(\PP^3)$
and the other is a 13 dimensional variety which parameterizes planar
cubics together with marked points.
\item $\bS(\PP^3)$ is isomorphic to
$\bH(\PP^3)$.\end{enumerate}\end{theo}

From these theorems, we obtain the following.
\begin{prop}
\begin{enumerate}
\item $\bH(\PP^r)$ is isomorphic to $\bH(\PP\cU)$ which is a
component of the relative Hilbert scheme for the bundle $\PP\cU\to
Gr(4,r+1)$ of $\PP^3$'s, where $\cU$ is the universal rank 4 vector
bundle on the Grassmannian $Gr(4,r+1)$. \item $\bH(\PP^r)$ is the
blow-up of $\bS(\PP^r)$ along the smooth locus of planar stable
sheaves. In particular, $\bS(\PP^r)$ is nonsingular.
\end{enumerate}
\end{prop}
\begin{proof}
(1) The relative Hilbert scheme $\cH ilb^{3m+1}(\PP U)$ over
$Gr(4,r+1)$ has fibers $\cH ilb^{3m+1}(\PP^3)$. By Theorem
\ref{PS1}, we obtain an irreducible component $\bH(\PP\cU)$ whose
fibers are $\bH(\PP^3)$. We then have a natural morphism
\[ \phi: \bH(\PP\cU)\to \bH(\PP^r)\]
sending $Z\subset \PP U$ to $Z\subset \PP^r$ via the inclusion $\PP
U\hookrightarrow \PP^r$ for any $U\in Gr(4,r+1)$. By Theorem
\ref{PS1}, $\bH(\PP\cU)$ is smooth and proper. Since both
$\bH(\PP\cU)$ and $\bH(\PP^r)$ are compactifications of the space of
smooth rational cubics in $\PP^r$, $\phi$ is birational.

In \cite{Harris}, all possible types of elements in $\bH(\PP^3)$ are
described. Looking at the list, one immediately sees that $\phi$ is
injective because the image $\phi(y)$ of every point $y\in
\bH(\PP\cU)$ determines a unique $\PP^3$. This certainly implies
that $\phi$ is bijective. By Zariski's main theorem, we deduce that
$\phi$ is an isomorphism.

(2) Similarly as above, we have a smooth family $\bS(\PP\cU)\to
Gr(4,r+1)$ whose fibers are $\bS(\PP U)$ for $U\in Gr(4,r+1)$.  By
Theorem \ref{FT1}, $\bH(\PP\cU)\cong \bS(\PP\cU)$ and thus it
suffices to study the relationship between $\bS(\PP\cU)$ and
$\bS(\PP^r)$. The inclusion $\imath:\PP U\hookrightarrow \PP^r$ for
$U\in Gr(4,r+1)$ induces a morphism
\[\psi:\bS(\PP\cU)\to \bS(\PP^r)\]
by sending $F$ to $\imath_*F$. This is an isomorphism on the
complement of the divisor $\Delta$ of planar sheaves. To describe
this divisor, let $\bS(\PP^2)\subset \cS imp^{3m+1}(\PP^2)$ be the
locus of sheaves $F$ each of which is a nontrivial extension $$0\to
\cO_C\to F\to \CC_p\to 0$$ for a singular planar cubic $C$ and a
singular point $p\in C$. See \cite[Lemma 2]{FT}. Then $\Delta$ is a
$\bS(\PP^2)$-bundle on $\PP \cU^*$ over $Gr(4,r+1)$. By Theorem
\ref{FT1}, $\Delta$ is smooth when $r=3$ and hence $\bS(\PP^2)$ is
nonsingular.

Using the isomorphism $\PP \cU^*\cong \PP(\CC^{r+1}/\cU')$ where
$\cU'$ is the universal rank 3 vector bundle on $Gr(3,r+1)$, we see
that $\Delta $ is isomorphic to $\bS(\PP
\cU')\times_{Gr(3,r+1)}\PP(\CC^{r+1}/\cU')$ over $Gr(3,r+1)$.
Obviously, the sheaf $F$ is independent of the choice of $\PP^3$
containing the support and hence $\psi$ is constant on the fibers
$\PP^{r-3}$ of $\PP(\CC^{r+1}/\cU')\to Gr(3,r+1)$. To deduce that
$\psi$ is the blow-up map along $\bS(\PP \cU')$ and hence
$\bS(\PP^r)$ is smooth, we only need to show that the normal bundle
of $\Delta$ restricted to a fiber of $\PP(\CC^{r+1}/\cU')\to
Gr(3,r+1)$ is $\cO_{\PP^{r-3}}(-1)$. But this can be easily checked.
For instance, suppose $C$ is the planar curve given by a map
$\PP^1\to \PP^r$
\[
(z_0^3:z_0^2z_1+z_0z_1^2:z_1^3:0:\cdots:0)
\]
and consider the family of smooth rational curves  given by
\[
(z_0^3:z_0^2z_1+z_0z_1^2:z_1^3:a_1z_0z_1^2:a_2z_0z_1^2:\cdots:a_{r-2}z_0z_1^2),\quad
(a_1,\cdots,a_{r-2})\in \CC^{r-2}-\{0\}.
\]
This gives us a morphism $\CC^{r-2}-\{0\}\to \bS(\PP^r)$ which has a
unique extension $\CC^{r-2}\to \bS(\PP^r)$, whose central image $F$
fits into a nonsplit exact sequence
\[
0\to \cO_C\to F\to \CC_p\to 0
\]
where $p$ is the unique nodal point. This extension is of course
obtained by taking the direct image $f_*\cO_{\PP^1}$ for each
$f:\PP^1\to \PP^r$ in the family parameterized by $\CC^{r-2}$. On
the other hand, we have a morphism $\CC^{r-2}-\{0\}\to \bS(\PP\cU)$
which extends to $\cO_{\PP^{r-3}}(-1)\to \bS(\PP\cU)$ by taking
direct image after choosing a $\PP^3$ containing the plane of $C$.
This means the normal bundle restricts to $\cO(-1)$ as desired.
Since $\Delta$ is flat over $\bS(\PP \cU')$, this holds for every
fiber.
\end{proof}

\section{From Kontsevich to Simpson}

In this section, we compare the Kontsevich compactification $\bM$
and the Simpson compactification $\bS$. We will prove that the
birational map $\bM\dashrightarrow \bS$ is the composition of three
blow-ups and three blow-downs, whose centers will be described
explicitly below.

Let us consider any family of stable maps of degree 3 parameterized
by a reduced scheme $Z$
\[\xymatrix{
\cC \ar[r]^\ev \ar[d]_\pi & \PP^r\\
Z }\] where $\pi$ is a family of connected curves of arithmetic
genus 0, with at worst nodal singularities. Let $\cE_0$ be the
direct image of $\cO_\cC$ by $(\ev,\pi):\cC\to \PP^r\times Z$.
Then $\cE_0$ is a family of coherent sheaves on $\PP^r$, flat over
$Z$ because the Hilbert polynomial is constantly $3m+1$. By Lemma
\ref{lem1}, the restriction of $\cE_0$ to $\PP^r\times Z_0$ where
$Z_0$ is the locus of smooth curves, is a family of stable sheaves
on $\PP^r$. Hence we obtain a birational map
\[
\Phi:\bM\dashrightarrow \bS .
\]
How do we eliminate the locus of indeterminacy? First we find the
locus of indeterminacy. Next we find a suitable sequence of
blow-ups and apply elementary modification to construct a family
of stable sheaves. Thus we get a morphism to $\bS$. Finally we
study the local geometry of the exceptional divisors and show that
we can contract the divisors. After these blow-downs, we obtain an
isomorphism of the resulting model of $\bM$ with $\bS$.

\bigskip

\subsection{Locus of indeterminacy}
To find the locus of indeterminacy of $\Phi$, we need the following
generalization of Lemma \ref{lem1}.
\begin{lemm}
If $f:C\to \PP^r$ is a stable map in $\bM$ with no multiple
components (i.e. no component of $f(C)$ is a multiply covered by
$f$), then $f_*\cO_C$ is a stable sheaf.
\end{lemm}
\begin{proof}
If $f$ is non-planar, $C'=f(C)$ is a Cohen-Macaulay curve and
$f_*\cO_{C}=\cO_{C'}$. By Lemma \ref{lem1}, $f_*\cO_C$ is stable.
Suppose now $f$ is planar. By our assumption, $C'$ is a singular
cubic and $f_*\cO_C$ fits into an exact sequence
\[
0\to \cO_{C'}\to f_*\cO_C\to \cO_p\to 0
\]
for some point $p\in C'$ because the Hilbert polynomial of
$\cO_{C'}$ is $3m$. The first map comes from the adjunction
$\cO_{\PP^r}\to f_*f^*\cO_{\PP^r}=f_*\cO_C$. Since $\Hom
(\cO_p,f_*\cO_C)=\Hom (f^*\cO_p,\cO_C)=0$, $f_*\cO_C$ is a
nonsplit extension. Because $C'$ is Cohen-Macaulay, $\cO_{C'}$ is
stable. It is easy to see now that $f_*\cO_C$ is stable.
\end{proof}

Consequently, the locus of unstable sheaves in the family $\cE_0$
has two irreducible components: one component $\Gamma^1$ consists of
stable maps, each of which has a single line as its image, and the
other component $\Gamma^2$ consists of stable maps, each of which
has a pair of lines as its image. Let $Gr(2,r+1)$ be the
Grassmannian of two dimensional subspaces of $\CC^{r+1}$ and $\cU$
be the universal rank 2 bundle. Let $\cM_0(\PP\cU,d)$  be the
relative moduli space of stable maps of degree $d$ and arithmetic
genus $0$ to the fibers of $\PP\cU\to Gr(2,r+1)$. The obvious map
$\cM_0(\PP\cU,3)\to Gr(2,r+1)$ is a locally trivial bundle with
fiber $\cM_0(\PP^1,3)$ and we have an inclusion map
\begin{equation}\lab{eq4.1}\Gamma^1:=\cM_0(\PP \cU,3)\hookrightarrow
\cM_0(\PP^r,3).\end{equation} By gluing operation at the marked
points, we see that the component $\Gamma^2$ is isomorphic to the
fiber product
\[
\cM_{0,1}(\PP\cU,2)\times_{\PP^r} \PP\cU
\]
where $\cM_{0,1}(\PP\cU,2)$ is the moduli space of stable maps of
genus $0$ and degree $2$ with one marked point to the fibers of
$\PP\cU\to Gr(2,r+1)$. Note here that
$\cM_{0,1}(\PP^r,1)=\PP\cU\cong \PP T_{\PP^r}$ and the marked points
give us morphisms to $\PP^r$ in the fiber product. 
For further analysis however, we need a description of $\Gamma^1$
and $\Gamma^2$ via GIT.

\bigskip

\subsection{Indeterminacy via GIT}
In \cite[\S5]{KM}, we proved that
\begin{equation}\lab{eq916-1} \bM\cong Q_0/ SL(2)\end{equation} for
a smooth quasi-projective variety $Q_0=\bP_5$. Let us briefly recall
the construction.

We start with the stable part $\bP_0$ of the projective space $\PP
(\Sym^3(\CC^2)\otimes \CC^{r+1})$ with respect to the action of
$SL(2)$ on $\Sym^3(\CC^2)$. An element in $\bP_0$ can be thought
of as an $(r+1)$-tuple $(f_0:\cdots :f_r)$ of homogeneous
polynomials in $z_0,z_1$ of degree 3. When there are no common
zeroes of these $(r+1)$ polynomials, we get a stable map. Let
$\pi_1:\bP_1\to \bP_0$ be the blow-up along the locus $\Sigma^3$
of three common zeroes. Next let $\pi_2:\bP_2\to \bP_1$ be the
blow-up along the proper transform of the locus $\bar\Sigma^2$ of
at least two common zeros. Let $\pi_3:\bP_3\to \bP_2$ be the
blow-up along the proper transform of the locus $\bar\Sigma^1$ of
at least one common zero. After these three blow-ups and
elementary modifications along the exceptional divisors, we obtain
a family of stable maps
\[
\xymatrix{\cC_3\ar[r]^{\ev_3}\ar[d] & \PP^r\\ \bP_3}
\]
parameterized by $\bP_3$ and thus a morphism $\bP_3\to \bM$.

The exceptional divisor of the second blow-up becomes a
$\PP^1$-bundle and the normal bundle restricted to each fiber
$\PP^1$ is $\cO(-1)$. Hence we can contract this divisor to obtain
$\pi_4:\bP_3\to \bP_4$. Then the exceptional divisor of the first
blow-up becomes a $\PP^2$-bundle and the normal bundle restricted to
each fiber $\PP^2$ is $\cO(-1)$. Hence we can contract this divisor
to obtain $\pi_5:\bP_4\to\bP_5$. The morphism $\bP_3\to \bM$ factors
through the two blow-downs and the induced map $\bP_5\to \bM$ is
$SL(2)$-invariant. Therefore, we obtain a morphism $\bP_5/SL(2)\to
\bM$ which turns out to be a bijection and hence an isomorphism.
Thus we obtain the isomorphism \eqref{eq916-1} with $Q_0:=\bP_5$.
Furthermore, by the same argument, $\cC_3$ can be blown down twice
and $\ev_3$ factors through
\[
\xymatrix{ \cC_3\ar[r]\ar[d] &\cC_4\ar[r]\ar[d] &\cC_5\ar@{=}[r]
\ar[d]&\cC
\ar[r]^{\ev}\ar[d]^{\pi} &\PP^r\\
\bP_3\ar[r] &\bP_4\ar[r] &\bP_5\ar@{=}[r] &Q_0 }
\]
Thus we have a natural morphism
$$\varphi_{\mathbf{f}}:\cO_{\PP^r\times Q_0}\to \mathbf{f}_*\cO_\cC=:\cE_0$$
where $\mathbf{f}=(\ev,\pi)$.

For a description of $\Gamma^1$ as the quotient of a smooth variety,
let us consider
\begin{equation}\lab{eq916-3}
\Theta^1:=\PP(\Sym^3(\CC^2)\otimes \CC^2)^s\times_{PGL(2)}
\PP(\CC^2\otimes \CC^{r+1})^s
\end{equation}
which is a $\PP(\Sym^3(\CC^2)\otimes \CC^2)^s$-bundle over
$Gr(2,r+1)$. The subscript $PGL(2)$ denotes the quotient by the
diagonal action of $PGL(2)$. Note that the stable part
$\PP(\Sym^3(\CC^2)\otimes \CC^2)^s$ is the stable part with
respect to the diagonal action on $\PP\Sym^3(\CC^2)^{\oplus 2}$
while the stable part $\PP(\CC^2\otimes \CC^{r+1})^s$ is with
respect to the diagonal action on $\PP\left( (\CC^2)^{\oplus
r+1}\right)$. We have $\Theta^1/ SL(2)\cong \cM_0(\PP\cU,3)$
because by \cite{KM}, \[ \PP(\Sym^3(\CC^2)\otimes
\CC^2)^s/SL(2)\cong \cM_0(\PP^1,3) .\] Further, the obvious
composition map
\begin{equation}\lab{eq916-2}\PP\Hom(\Sym^3(\CC^2), \CC^2)\times
\PP\Hom(\CC^2,\CC^{r+1})\lra \PP\Hom(\Sym^3(\CC^2),
\CC^{r+1})\end{equation} induces a morphism $\Theta^1\to \bP_0$.
It is straightforward to keep track of the blow-ups and -downs of
$\bP_0$ and see that we obtain an injective morphism
\[
\Theta^1\hookrightarrow Q_0
\]
whose quotient gives us the embedding $\Gamma^1\hookrightarrow \bM$.

Now we provide a description of the other component of the
indeterminacy locus via GIT. The set of homogeneous polynomials of
degree $1$, up to constant multiple, is $\PP^1$. Multiplication of
polynomials and composition \eqref{eq916-2} give us
\begin{equation}\lab{eq916-4}
[\PP^1\times \PP(\Sym^2(\CC^2)\otimes
\CC^2)]^s\times_{PGL(2)}\PP(\CC^2\otimes\CC^{r+1})^s\lra
\PP(\Sym^3(\CC^2)\otimes \CC^{r+1})^s=\bP_0 \end{equation} where
the stability of the first factor is with respect to the diagonal
action on $\cO(1,1)$ of $\PP(\CC^2)\times \PP
\Sym^2(\CC^2)^{\oplus 2}$ while the stability on the second factor
is with respect to the diagonal $SL(2)$-action on $\PP\left(
(\CC^2)^{\oplus r+1}\right)$. It is also a straightforward
exercise to keep track of this morphism through the blow-ups and
-downs of $\bP_0$. The first blow-up $\pi_1$ corresponds to the
blow-up along
\begin{equation}\lab{eq916-7}(\PP^1\times
(\PP(\Sym^2\CC^2)\times\PP^1))^s\times_{PGL(2)}\PP(\CC^2\otimes\CC^{r+1})^s.\end{equation}
The second blow-up $\pi_2$ is the identity map because the blow-up
center is a smooth divisor. Then as in the proof of \cite[Lemma
5.4]{KM}, we obtain an equivariant embedding
\begin{equation}\lab{eq916-6}
[\PP^1\times bl_{\PP^2\times\PP^1}\PP(\Sym^2(\CC^2)\otimes
\CC^2)]^s\times_{PGL(2)}\PP(\CC^2\otimes\CC^{r+1})^s\hookrightarrow
\bP_2 .
\end{equation}
Since the image of \eqref{eq916-6} is contained in the blow-up
center, the third blow-up $\pi_3$ gives us a $\PP^{r-1}$-bundle
over the image of \eqref{eq916-6}. By the construction of stable
maps \cite[\S5]{KM}, this bundle parameterizes all possible stable
maps whose images are two lines. The first blow-down doesn't make
any change while the second blow-down contracts the exceptional
divisor of the first blow-up. Therefore, we obtain a
$\PP^{r-1}$-bundle over $$[\PP^1\times \PP(\Sym^2(\CC^2)\otimes
\CC^2)]^s\times_{PGL(2)}\PP(\CC^2\otimes\CC^{r+1})^s$$ which we
denote by $\Theta^2$. Thus, we obtain a smooth subvariety
\[
\Theta^2\hookrightarrow Q_0
\]
whose quotient gives us the locus of stable maps $\Gamma^2$ with
bilinear image.

Let us now consider the normal bundle of $\Theta^1$. The normal
bundle of a fiber $\PP(\Sym^3(\CC^2)\otimes \CC^2)^s$ of $\Theta^1$
over $Gr(2,r+1)$ in $Q_0$ is the pull-back of the normal bundle of
$\cM_0(\PP^1,3)$ in $\bM=\cM_0(\PP^r,3)$ i.e.
$$\pi_*\ev^*N_{\PP^1/\PP^r}=(\pi_*\ev^*\cO_{\PP^1}(1))^{\oplus r-1}$$
where
\begin{equation}\label{eqb-1}
\xymatrix{\cC\ar[r]^{\ev}\ar[d]_\pi & \PP^1\\
\PP(\Sym^3(\CC^2)\otimes \CC^2)^s}
\end{equation}
is the family of stable maps to $\PP^1$. \begin{lemm}
$$\pi_*\ev^*\cO_{\PP^r}(1)\cong \cO_{\PP^7}^{\oplus 2}\oplus
\cO_{\PP^7}(-1)^{\oplus 2}.
$$\end{lemm}
\begin{proof}
Let $P_0=\PP(\Sym^3(\CC^2)\otimes \CC^2)^s$. Since the complement of
$P_0$ has codimension 3, the vector bundle $\pi_*ev^*\cO(1)$ on
$P_0$ extends uniquely to a vector bundle $E$ on $\PP^7$. By the
splitting criterion of Horrocks \cite[Theorem 2.3.1]{OSS}, it
suffices to prove that \begin{enumerate}\item $H^i(\PP^7,E(k))=0$
for $i=1,2,\cdots, 6$ and $k\in \ZZ$.
\item The Hilbert polynomial $H^0(\PP^7,E(k))$ of $E$ coincides with that
of $\cO_{\PP^7}^{\oplus 2}\oplus \cO_{\PP^7}(-1)^{\oplus 2}.$
\end{enumerate}
The first condition tells us that the bundle $E$ splits into a
direct sum of line bundles and the second condition shows that the
line bundles are two $\cO$s and two $\cO(-1)$s.

By \cite[Lemma 5.3]{KM}, the locus $\bar\Sigma^1$ of at least one
common zeros is a divisor in $P_0$ and the locus of indeterminacy of
the birational map \cite[(5.30)]{KM}
$$\PP^1\times P_0\dashrightarrow \PP^1$$
is the normalization $\tilde\Sigma$ of $\bar\Sigma^1$, which is
isomorphic to the stable part of $\PP^1\times
\PP(\Sym^2(\CC^2)\otimes\CC^2)$ with respect to the linearization
$\cO(1,1)$. Hence $\cC$ is the blow-up of $\PP^1\times P_0$ along
$\tilde\Sigma$ and the family of stable maps \eqref{eqb-1} is
obtained by the surjective homomorphism $$\cO_{\cC}^{\oplus 2}\lra
\mu^*\cO_{\PP^1\times P_0}(3,1)\otimes
\cO(-\mu^{-1}(\tilde\Sigma))=:H$$ induced from the evaluation
$\cO_{\PP^1\times P_0}^{\oplus 2}\to \cO_{\PP^1\times P_0}(3,1)$.
Here $\mu:\cC\to \PP^1\times P_0$ is the blow-up map. See
\cite{KM} for more details.

By construction, $\ev^*\cO_{\PP^1}(1)=H$ and from the exact sequence
\[
0\lra H\lra \mu^*\cO(3,1)\lra
\mu^*\cO(3,1)|_{\mu^{-1}(\tilde\Sigma)}\lra 0
\]
we obtain an exact sequence on $\PP^1\times P_0$
\[
0\lra \mu_*H\lra \cO(3,1)\lra \cO(3,1)|_{\tilde\Sigma}\lra 0.
\]
Since the pullback of $\cO_{P_0}(1)$ to $\tilde\Sigma=[\PP^1\times
\PP(\Sym^2(\CC^2)\otimes\CC^2)]^s$ is $\cO(1,1)$, we have
$$\cO(3,1)|_{\tilde\Sigma}\cong \cO_{\PP^1\times\PP^5}(4,1).$$
Let $\nu:\tilde\Sigma\to \bar\Sigma^1$ be normalization map. Then by
taking the direct images with respect to the projection to $P_0$, we
obtain an exact sequence
\[
0\lra \pi_*\ev^*\cO_{\PP^1}(1)\lra \cO_{P_0}(1)^{\oplus 4}\lra
\nu_*(\cO(3,1)|_{\tilde\Sigma})\lra 0
\]
because $R^1\pi_*\ev^*\cO_{\PP^1}(1)=0$. After tensoring with
$\cO_{\PP^7}(k)$, we obtain
\[
0\lra \pi_*\ev^*\cO_{\PP^1}(1)\otimes \cO_{\PP^7}(k)\lra
\cO_{P_0}(k+1)^{\oplus 4}\lra
\nu_*(\cO_{\PP^1\times\PP^5}(k+4,k+1))\lra 0.
\]
Since the codimension of $\PP^7-P_0$ is 3, we thus obtain the long
exact sequence
\[
0\lra H^0(\PP^7,E(k))\lra H^0(\PP^7,\cO_{\PP^7}(k+1))^{\oplus
4}\mapright{\gamma} H^0(\PP^1\times\PP^5,\cO(k+4,k+1)) \]
\[
\lra H^1(\PP^7,E(k))\lra H^1(\PP^7,\cO_{\PP^7}(k+1))^{\oplus 4}\lra
\cdots
\]
Clearly, $H^i(\PP^7,\cO_{\PP^7}(k+1))^{\oplus 4}=0$ for
$i=1,2,\cdots, 6$ and $H^i(\PP^1\times\PP^5,\cO(k+4,k+1))=0$ for
$i=1,2,\cdots,5$. Also, by an elementary calculation, we have
\[
\dim H^0(\PP^7,\cO_{\PP^7}(k+1))^{\oplus 4} -\dim
H^0(\PP^1\times\PP^5,\cO(k+4,k+1))= \dim H^0(\PP^7,\cO(k)\oplus
\cO(k-1))^{\oplus 2}.
\]
Therefore, the lemma follows if $\gamma$ is surjective. Indeed, the
space $H^0(\PP^7,\cO_{\PP^7}(k+1))$ consists of homogeneous
polynomials of degree $k+1$ in eight variables $x_1,\cdots, x_8$ and
$H^0(\PP^1\times\PP^5,\cO(k+4,k+1))$ consists of bihomogeneous
polynomials in $z_0,z_1$ and $y_1,\cdots, y_6$ of bidegree
$(k+4,k+1)$. From the definition of the map $\nu$, we see that
\[ x_1=z_0y_1, x_2=z_0y_2+z_1y_1, x_3=z_0y_3+z_1y_2, x_4=z_1y_3,\]
\[ x_5=z_0y_4, x_6=z_0y_5+z_1y_4, x_7=z_0y_6+z_1y_5, x_8=z_1y_6 .\]
By induction on $k$, it is easy to check that $\gamma$ is indeed a
surjection.
\end{proof}

Consequently, the normal bundle in $Q_0$ of
$\PP(\Sym^3(\CC^2)\otimes \CC^2)^s$, which is a fiber of
$\Theta^1\to Gr(2,r+1)$, is
$$\cO_{\PP^7}^{\oplus 2r-2}\oplus \cO_{\PP^7}(-1)^{\oplus 2r-2}.$$
Obviously the factor $\cO^{\oplus 2r-2}$ is the pullback of the
tangent space of $Gr(2,r+1)$ and hence the normal bundle to
$\Theta^1$ in $Q_0$ restricted to $\PP(\Sym^3(\CC^2)\otimes
\CC^2)^s$ is
\[
\cO_{\PP^7}(-1)^{\oplus 2r-2}.
\]

We summarize the above discussions as follows.
\begin{coro}\label{cor-b3}
The indeterminacy locus of $\Phi$ is $\Gamma^1\cup \Gamma^2$ with
$\Gamma^1=\Theta^1/SL(2)$, $\Gamma^2=\Theta^2/SL(2)$ where
$\Theta^1$ is a $\PP(\Sym^3(\CC^2)\otimes \CC^2)^s$-bundle over
$Gr(2,r+1)$ and $\Theta^2$ is a $\PP^{r-1}$-bundle over a
$[\PP^1\times \PP(\Sym^2(\CC^2)\otimes \CC^2)]^s$-bundle over
$Gr(2,r+1)$. The normal bundle to $\Theta^1$ in $Q_0$ restricted to
$\PP(\Sym^3(\CC^2)\otimes \CC^2)^s$ is $\cO_{\PP^7}(-1)^{\oplus
2r-2}.$
\end{coro}

\bigskip

\subsection{Blow-ups}
By adjunction, we have a natural homomorphism
$\varphi_f:\cO_{\PP^r}\to f_*\cO_C$ for any stable map $f:C\to
\PP^r$. The image of $\varphi_f$ is the structure sheaf $\cO_{C'}$
of a curve $C'$ in $\PP^r$. If $f\in \Gamma^1\cup \Gamma^2$,
$\cO_{C'}$ is a destabilizing subsheaf of $f_*\cO_C$ because \[
\frac{\chi(\cO_{C'}(m))}{r(\cO_{C'})}=m+1\text{ or }\frac{2m+1}2 >
\frac{3m+1}3=\frac{\chi(f_*\cO_C(m))}{r(f_*\cO_C)}\]
\begin{defi} A subsheaf $F$ of a coherent sheaf $E$ on $\PP^r$ is
the \emph{destabilizing subsheaf} if it is the first nonzero term
$E_1$ in the Harder-Narasimhan filtration $0\subset E_1\subset
\cdots\subset E_k=E$ and $F\ne E$. The quotient $E/F$ by the
destabilizing subsheaf $F$ is called the \emph{destabilizing
quotient}.
\end{defi} See \cite{HL} for fundamental results on
Harder-Narasimhan filtration.

\begin{lemm} $\cO_{C'}$ is the destabilizing subsheaf of $f_*\cO_C$
when $f_*\cO_C$ is not stable, i.e. $f\in \Gamma^1\cup \Gamma^2$.
\end{lemm}
\begin{proof}
For $f\in \Gamma^1$, $f_*\cO_C$ is supported on a line $L$ in
$\PP^r$ and hence $f_*\cO_C\cong \cO_L(a)\oplus \cO_L(b)\oplus
\cO_L(c)$ for $a\ge b\ge c$. Since the Hilbert polynomial of
$f_*\cO_C$ is $3m+1$, $a+b+c=-2$. Also, $f_*\cO_C$ admits a unique
section $\varphi_f$ and hence $a=0,b=-1,c=-1$. Therefore, the image
$\cO_{C'}$ of $\varphi_f$ is the destabilizing subsheaf $\cO_L$.

For $f\in \Gamma^2-\Gamma^1$, $C'$ is the union of two lines $L_1$
and $L_2$ such that $f_*\cO_C|_{L_1}\cong \cO_{L_1}$ and
$f_*\cO_C|_{L_2}\cong \cO_{L_2}\oplus \cO_{L_2}(-1)$. As $\cO_{C'}$
is the gluing of $\cO_{L_1}$ and $\cO_{L_2}$ at a point, this is
certainly the destabilizing subsheaf of $f_*\cO_C$.
\end{proof}

From the above proof, we see that when $f\in \Gamma^1$,
\begin{equation}\label{eq-b4}
f_*\cO_C\cong \cO_L\oplus \cO_L(-1)\oplus \cO_L(-1)\end{equation}
for a line $L=f(C)$ in $\PP^r$. Note that the direct image sheaf
depends only on the image line $L$.

Let $q_1:Q_1\to Q_0$ be the blow-up along the smooth subvariety
$\Theta^1$ and let $p_1:\bM_1\to \bM$ be the quotient, i.e. the
blow-up along the subvariety $\Gamma^1$ (\cite[Lemma 3.11]{K4}). Let
$\Theta_1^1$ be the exceptional divisor and $\Theta_1^2$ be the
proper transform of $\Theta^2$. Let $\Gamma_1^i$ be the quotient of
$\Theta_1^i$ by $SL(2)$. We pull back $\varphi_{\mathbf{f}}$ to
$\PP^r\times Q_1$ and restrict it to the divisor $\PP^r\times
\Theta^1_1$. Let $A_1$ be its cokernel on $\PP^r\times \Theta^1_1$.
Let $\cE_1$ be the kernel of the composition
\[
(1\times q_1)^*\cE_0\lra (1\times q_1)^*\cE_0|_{\PP^r\times
\Theta^1_1}\lra A_1
\]

\begin{lemm}\lab{lem4}
\begin{enumerate}
\item $\cE_1$ is a flat family of
coherent sheaves on $\PP^r$ parameterized by $Q_1$.
\item The locus where $\cE_1$ is not stable has two irreducible
components $\Theta_1^2\cup \Theta^3_1$ where $\Theta_1^3$ is a
smooth subvariety of $\Theta^1_1$, which is a $\PP^1\times
\PP^{r-2}$-bundle over $\Theta^1$.
\end{enumerate}
\end{lemm}
\begin{proof}
Let $f:C\to L\subset \PP^r$ be a point in $\Theta^1$. If we denote
the graph of $f$ in $C\times \PP^r$ by $G$, it is well known that
the deformation space of the map $f$ with $C$ fixed is the same as
the tangent space at $G$
$$\Hom_{C\times\PP^r} (I_G,\cO_G)$$ of the Hilbert scheme of closed
subschemes in $C\times \PP^r$ where $I_G$ is the ideal sheaf of $G$.
From the short exact sequence
\[ 0\lra I_G\lra \cO_{C\times\PP^r}\lra \cO_G\lra 0\]
we obtain an isomorphism
\[ \Hom(I_G,\cO_G)\cong \Ext^1_{C\times \PP^r}(\cO_G,\cO_G)\]
because $H^1(\cO_G)=H^1(\cO_C)=0$. Since $I_G|_G\cong
f^*\Omega_{\PP^r}$ and $\cO_G\cong \cO_C$, we have isomorphisms
$$\Hom(I_G,\cO_G)\cong \Hom_C(f^*\Omega_{\PP^r}, \cO_C)\cong
\Hom_{\PP^r}(\Omega_{\PP^r}, f_*\cO_C).$$ The natural morphism
$I_L\mapright{d} \Omega_{\PP^r}$ and the projection $f_*\cO_C\to
\cO_L(-1)^{\oplus 2}$ by \eqref{eq-b4} induce the homomorphisms \[
\Hom(\Omega_{\PP^r},f_*\cO_C)\lra \Hom (I_L,f_*\cO_C)\lra
\Hom(I_L,\cO_L(-1))^2 \] where $I_L$ is the ideal sheaf of $L$ on
$\PP^r$. On the other hand, as $R^1p_*\cO_G=0$ where
$p:C\times\PP^r\to \PP^r$ is the projection, an extension of $\cO_G$
by itself gives us an extension of $p_*\cO_G\cong f_*\cO_C$ by
itself. Thus we have homomorphisms
\[
\Ext^1(\cO_G,\cO_G)\lra \Ext^1(f_*\cO_C,f_*\cO_C)\lra
\Ext^1(\cO_L,f_*\cO_C)\lra \Ext^1(\cO_L,\cO_L(-1))^2
\]
where the last two come from \eqref{eq-b4}. Then it is an easy
exercise to check the commutativity of the following diagram
\begin{equation}\label{eq-b33}
\xymatrix{ \Hom(I_G,\cO_G) \ar[r]^\cong \ar[d]_\cong & \Ext^1(\cO_G,\cO_G)\ar[d]\\
\Hom(\Omega_{\PP^r}, f_*\cO_C)\ar[d] & \Ext^1(f_*\cO_C,f_*\cO_C)\ar[d]\\
\Hom(I_L,f_*\cO_C)\ar[r]^\cong \ar[d] &
\Ext^1(\cO_L,f_*\cO_C)\ar[d]\\
\Hom(I_L,\cO_L(-1))^2\ar[r]^\cong & \Ext^1(\cO_L,\cO_L(-1))^2
}\end{equation} where the last two horizontal maps are from the
exact sequence $0\to I_L\to \cO_{\PP^r}\to \cO_L\to 0$.

From the exact sequence $$0\lra I_L/I_L^2\lra \Omega_{\PP^r}|_L\lra
\Omega_L\lra 0$$ we see that the normal space to the deformation
space of $f$ as a map to $L$ in the deformation space of $f$ as a
map to $\PP^r$ is exactly $\Hom(I_L,f_*\cO_C)$. Furthermore, the
tangent space to the Grassmannian is exactly $\Hom(I_L,\cO_L)$ and
hence the normal space of $\Theta^1$ in $Q_0$ at $f$ is
$$N_{\Theta^1/Q_0,f}=\Hom(I_L,\cO_L(-1))^2$$
which is the bottom left term in \eqref{eq-b33}.

The Kodaira-Spencer map for the family $\cE_0$ of sheaves
\begin{equation}\label{eq-b7} T_fQ_0\to
\Ext^1_{\PP^r}(f_*\cO_C,f_*\cO_C)\end{equation} sends any extension
$ \tilde f:\tilde
C=C\times\mathrm{Spec}\CC[\epsilon]/(\epsilon^2)\lra \PP^r $ of $f$
to the extension class of $ 0\lra \epsilon \cdot f_*\cO_C\lra \tilde
f_*\cO_{\tilde C}\lra f_*\cO_C\lra 0 $ obtained from the obvious
extension $0\lra \epsilon\cO_C\lra \cO_{\tilde C}\lra \cO_C\lra 0$
since $R^1f_*\cO_C=0$. By \eqref{eq-b4}, $f_*\cO_C$ remains fixed
along the fibers of $\Theta^1\to Gr(2,r+1)$ and the variation of $L$
is sent to the factor $\Ext^1(\cO_L,\cO_L)$. Hence we have an
induced homomorphism
\begin{equation}\label{eq-b5}
N_{\Theta^1/Q_0,f}=\Hom(I_L,\cO_L(-1))^2\lra
\Ext^1(\cO_L,\cO_L(-1))^2
\end{equation}
by the projection $\Ext^1_{\PP^r}(f_*\cO_C,f_*\cO_C)\to
\Ext^1(\cO_L,\cO_L(-1))^2.$ From the above discussions, we see that
the horizontal maps in \eqref{eq-b33} give us the Kodaira-Spencer
map and hence \eqref{eq-b5} is an isomorphism.

Let us consider the effect of elementary modification $\cE_1$.
Choosing a point in $Q_1$ lying over $f$ is the same as choosing a
normal vector to $\Theta^1$ at $f$. Over $\Spec
\CC[\epsilon]/(\epsilon^2)$, the process of taking $\cE_1$ from
$\cE_0$ is explained in the diagram
\[
\xymatrix{
& & 0 \ar[d] & 0 \ar[d] \\
0 \ar[r] & {\epsilon\cdot f_*\cO_C} \ar[r]\ar@{=}[d] & \cE_1
\ar[d]\ar[r]
&\cO_L \ar[r] \ar[d] &0\\
0\ar[r] &\epsilon\cdot f_*\cO_C\ar[r] & \cE_0 \ar[d]\ar[r]
&f_*\cO_C \ar[r] \ar[d] &0\\
& & A_1 \ar@{=}[r] \ar[d] & \cO_L(-1)^{\oplus 2} \ar[d] \\ & & 0 & 0
}
\]
Since $f_*\cO_C=\cO_L\oplus \cO_L(-1)^{\oplus 2}$, the central fiber
$\cE_1/\epsilon \cE_1$ fits into an exact sequence
\[
0\lra \cO_L(-1)^{\oplus 2} \lra \cE_1/\epsilon \cE_1\lra \cO_L\lra 0
\]
whose extension class is the image of \eqref{eq-b5}. In particular,
$f_*\cO_C\cong \cO_L\oplus \cO_L(-1)\oplus \cO_L(-1)$ and
$\cE_1|_{\PP^r\times\{y\}}$ is an extension of $\cO_L$ by
$\cO_L(-1)\oplus \cO_L(-1)$ for any $y\in Q_1$ lying over $f$. In
particular, the Hilbert polynomial is constantly $3m+1$ and hence
$\cE_1$ is flat over $Q_1$. Each element in $\PP
N_{\Theta^1/Q_0,f}\cong \PP \left(
\Ext^1_{\PP^r}(\cO_L,\cO_L(-1))\otimes \CC^2\right) $ parameterizes
an extension
\[ 0\to \cO_L(-1)\oplus \cO_L(-1)\to E\to \cO_L\to 0 \]
which is unstable if and only if it belongs to $\PP
\Ext^1_{\PP^r}(\cO_L,\cO_L(-1))\times \PP^1$, in which case
$E=F\oplus \cO_L(-1)$ for a stable sheaf $F$ which is a nontrivial
extension of $\cO_L$ by $\cO_L(-1)$.
\end{proof}

\begin{exam}\emph{
Let $f:\PP^1\to \PP^r$ be a stable map given by
$(z_0^3:z_1^3:0:\cdots:0)$ and consider a deformation of $f$ over
$\Spec \CC[\epsilon]/(\epsilon^2)$ given by the map
\[
(z_0^3:z_1^3:\epsilon z_0^2z_1:\epsilon  z_0z_1^2:0:\cdots :0).
\]
In affine charts, $\cE_0=f_*\cO_C$ is $\CC[t,\epsilon]$ which is a
module over $\CC[x_1,x_2,\cdots,x_r]$ by \[ x_1=t^3,\quad
x_2=\epsilon t,\quad x_3=\epsilon t^2,\quad x_4=\cdots =x_r=0\]
Then the central fiber of $\cE_0$ is $$\cE_0/\epsilon
\cE_0=\CC[t]\cong \CC[x_1]\oplus t\CC[x_1]\oplus t^2\CC[x_1]$$ and
the destabilizing subsheaf is $\CC[x_1]$. Hence we have
\[ \cE_1=\ker (\CC[t,\epsilon]\lra \CC[t]\lra t\CC[x_1]\oplus
t^2\CC[x_1])=\{g(x_1)+\epsilon h(t,\epsilon)\}
\]
where $g,h$ are polynomials. Hence the central fiber of $\cE_1$ is
\[
\cE_1/\epsilon \cE_1=\{g(x_1)+\epsilon t g_1(x_1) +\epsilon t^2
g_2(x_1)\} \cong \CC[x_1]\oplus x_2\CC[x_1]\oplus x_3\CC[x_1]\]
\[\cong
\CC[x_1,x_2,\cdots,x_r]/(x_2^2,x_2x_3,x_3^2, x_4,\cdots, x_r)
\]
i.e. the thickening of the line $x_2=x_3=\cdots =0$ in $\PP^3$ given
by $x_4=\cdots=0.$}\end{exam}
\begin{rema}\label{rem4.8}\emph{ By Corollary \ref{cor-b3}, $\Theta^1_1$ is a
$\PP^{2r-3}\times (\PP^7)^s$ bundle over $Gr(2,r+1)$. The proof of
Lemma \ref{lem4} tells us that the family of stable sheaves $\cE_1$
over $(\PP^{2r-3}-\PP^1\times\PP^{r-2})\times (\PP^7)^s$ remains
constant for the $(\PP^7)^s$ direction and depends only on the
extension classes in $\PP^{2r-3}=\PP\Ext^1(\cO_L,\cO_L(-1))^2$.
}\end{rema}

By Lemma \ref{lem4}, we have an invariant morphism
\[
\Psi_1:Q_1-\Theta_1^2\cup\Theta_1^3\lra \bS
\]
which induces a morphism
\[
\Phi_1:\bM_1-\Gamma_1^2\cup\Gamma_1^3\lra \bS
\]
where $\Gamma_1^i$ is the quotient of $\Theta^i_1$.

Let $q_2:Q_2\to Q_1$ be the blow-up along $\Theta_1^2$ and
$p_2:\bM_2\to \bM_1$ be its quotient by $SL(2)$. Let $\Theta_2^2$ be
the exceptional divisor and $\Theta_2^1, \Theta_2^3$ be the proper
transforms of $\Theta_1^1, \Theta_1^3$. Let $\Gamma_2^i$ be the
quotient of $\Theta_2^i$ by $SL(2)$. Let $\cE_2'$ be the pull-back
of $\cE_1$ to $\PP^r\times Q_2$.
\begin{lemm} There is a unique flat family of quotient sheaves on
the divisor $\Theta_2^2$
\[
\cE_2'|_{\PP^r\times \Theta^2_2}\to A_2
\]
such that the Hilbert polynomial of $A_2$ at every point of
$\Theta_2^2$ is $m$.\end{lemm} \begin{proof} To prove this claim, it
suffices to show that for any $y\in \Theta^2_2$,
$\cE_2'|_{\PP^r\times\{y\}}$ has a destabilizing subsheaf of Hilbert
polynomial $2m+1$. From the uniqueness of the Harder-Narasimhan
filtration and the existence of relative Quot scheme \cite[Chapter
2]{HL}, we deduce that there is such a flat quotient $A_2$.

Let $f\in \Theta^2-\Theta^1$, i.e. $f:C\to \PP^r$ is a stable map
whose image is the union $C'$ of two distinct lines $L_1$ and
$L_2$. Without loss of generality, we may assume $L_2$ is the
degree 2 component. By adjunction, we have a subsheaf $\cO_{C'}$
of $f_*\cO_C$ and a nonsplit extension
\begin{equation}\lab{eq4.5} 0\to \cO_{C'}\to f_*\cO_C\to
\cO_{L_2}(-1)\to 0.\end{equation} Since $\cO_{C'}$ and
$\cO_{L_2}(-1)$ are stable, the destabilizing subsheaf of
$f_*\cO_C$ is $\cO_{C'}$, whose Hilbert polynomial is $2m+1$.

From Lemma \ref{lem4}, we see that $\Theta_1^2\cap
\Theta_1^1\subset \Theta_1^3$. Suppose $y\in \Theta_1^2\cap
\Theta_1^3$. Then $q_1(y)$ is represented by a stable map $f:C\to
L\subset \PP^r$ where $C=C_1\cup C_2$ is reducible and $L$ is a
line. From the proof of Lemma \ref{lem4}, one can check that
$q_1^{-1}(f)\cap \Theta_1^2=\PP \Ext^1(\cO_L,\cO_L(-1))$ and
$\cE_1|_{\PP^r\times \{y\}}\cong F\oplus \cO_L(-1)$ where $F$ is a
nonsplit extension of $\cO_L$ by $\cO_L(-1)$. Hence, $F$ is the
destabilizing subsheaf of $\cE_1|_{\PP^r\times \{y\}}$ whose
Hilbert polynomial is $2m+1$.
\end{proof}

Let $\cE_2$ be the kernel of the composition
\[
\cE_2'\twoheadrightarrow \cE'_2|_{\PP^r\times
\Theta_2^2}\twoheadrightarrow A_2.
\]
\begin{lemm}\lab{lem5}\begin{enumerate}
\item $\cE_2$ is a flat family of coherent sheaves on $\PP^r$
parameterized by $Q_2$.
\item The locus of unstable sheaves is precisely $\Theta_2^3$ and we
have an invariant morphism
\[\Psi_2:Q_2-\Theta_2^3\lra \bS\]
which induces a morphism
\[ \Phi_2:\bM_2-\Gamma_2^3\lra \bS.\]
\end{enumerate}
\end{lemm}
\begin{proof} As we saw in the proof of Lemma \ref{lem4}, the effect
of elementary modification is interchanging the destabilizing
subsheaf with the quotient. For $y\in
\Theta^2_2-q_2^{-1}(\Theta_1^3)$, $\cE_2|_{\PP^r\times\{y\}}$ is
an extension
\begin{equation}\lab{eq4.6}
0\to \cO_{L_2}(-1)\to \cE_2|_{\PP^r\times\{y\}}\to \cO_{C'}\to 0
\end{equation}
and thus the Hilbert polynomial remains unchanged. We claim
$E=\cE_2|_{\PP^r\times\{y\}}$ is stable. Indeed, \eqref{eq4.5} gives
an exact sequence
\[
0\to N_{C'}\to N_{C'}\otimes f_*\cO_C\cong f_*f^*N_{C'}\to
\cO_{L_2}(-1)\otimes N_{C'}\to 0
\]
by tensoring the normal bundle $N_{C'}$ of the complete intersection
$C'$. Upon taking cohomology, we get a diagram
\begin{equation}\label{304-1} \xymatrix{ 0\ar[r]& H^0(C',N_{C'})\ar[r] &
H^0(C',f_*f^*N_{C'})\ar[r] & H^0(C',
\cO_{L_2}(-1)\otimes N_{C'})\ar[r] & 0
} \end{equation} 

The deformation of $C'$ while keeping a node is parameterized by the
kernel of
$$H^0(C',N_{C'})\lra \Ext^1(\Omega_{C'},\cO_{C'})\cong \CC$$
which comes from the exact sequence
\[
0\to N_{C'}^*\to \Omega_{\PP^r}|_{C'}\to \Omega_{C'} \to 0.
\]
Here, $\Ext^1(\Omega_{C'},\cO_{C'})\cong\CC$ is the smoothing
direction of the node of $C'$.

Next, we calculate the quotient
\[ T_f\cM_0(\PP^r,3)/T_f\cM_0(C',3).
\]
It is well known that the tangent space of $\cM_0(\PP^r,3)$ (resp.
$\cM_0(C',3)$) is given by the hypercohomology
\[
\Ext^1(\{f^*\Omega_{\PP^r}\to \cO_C\},\cO_C)\quad (\text{resp. }
\Ext^1(\{f^*\Omega_{C'}\to \cO_C\},\cO_C) ).
\]
By applying the octahedron axiom for the derived category of sheaf
complexes on $C$ to the composition $f^*\Omega_{\PP^r}\to
f^*\Omega_{C'}\to  \cO_C$, we obtain a distinguished triangle
\[
f^*N_{C'}^*[1]\lra \{f^*\Omega_{\PP^r}\to \cO_C\}\lra
\{f^*\Omega_{C'}\to \cO_C\} \mapright{[1]}.
\]
This induces an exact sequence
\[
0\lra T_f\cM_0(C',3)\lra T_f\cM_0(\PP^r,3)\lra H^0(C,f^*N_{C'})\lra
\Ext^2(\{f^*\Omega_{C'}\to \cO_C\},\cO_C)\lra 0.
\]
Since the dimension of $\cM_0(C',3)$ (resp. $\cM_0(\PP^r,3)$) at $f$
is 2 (resp. $4r$) and the dimension of $H^0(C,f^*N_{C'})$ is
$7+4(r-2)$ by $N_{C'}=\cO(2)\oplus \cO(1)^{r-2}$, we have
$ \Ext^2(\{f^*\Omega_{C'}\to \cO_C\},\cO_C)\cong \CC.$ 
Thus we have a commutative diagram
\[
\xymatrix{ 0\ar[r] &H^0(C',T_{C'})\ar[r]\ar[d] &
H^0(C',T_{\PP^r})\ar[d]\ar[r] & H^0(C',N_{C'})\ar[d]\ar[r] &\CC
\ar[r] &0\\
0\ar[r] &H^0(C',f_*f^*T_{C'})\ar[r]\ar[d]^\cong &
H^0(C',f_*f^*T_{\PP^r})\ar[d]^\cong\ar[r] &
H^0(C',f_*f^*N_{C'})\ar[d]^{\cong}\\
0\ar[r] &H^0(C,f^*T_{C'})\ar[r]\ar[d] &
H^0(C,f^*T_{\PP^r})\ar[d]\ar[r] &
H^0(C,f^*N_{C'})\ar@{=}[d]\\
0\ar[r] & T_f\cM_0(C',3)\ar[r] & T_f\cM_0(\PP^r,3) \ar[r]
&H^0(C,f^*N_{C'})\ar[r] & \CC\ar[r]& 0, }\] which induces a
commutative diagram
\[
\xymatrix{ H^0(C',N_{C'})\ar[r]\ar[d] &\CC\ar@{=}[d]\\
H^0(C,f^*N_{C'})\ar[r] &\CC. }\] The kernel of the top row is the
deformation of $C'$ and the kernel of the bottom row is the quotient
$ T_f\cM_0(\PP^r,3)/T_f\cM_0(C',3).$ Therefore, by \eqref{304-1} the
normal space to $\Theta^2_1$ at $q_2(y)$ in $Q_1$ is isomorphic to
\[
H^0(C', \cO_{L_2}(-1)\otimes N_{C'})\cong \Hom (\cI_{C'},
\cO_{L_2}(-1))\cong \Ext^1(\cO_{C'},\cO_{L_2}(-1))
\]
where $\cI_{C'}$ is the ideal sheaf of $C'$. As in the proof of
Lemma \ref{lem4}, one can check that this isomorphism is compatible
with the Kodaira-Spencer map for $\cE_1$ and thus \eqref{eq4.6} is
nonsplit. Stability follows immediately.

Next suppose $y\in q_2^{-1}(\Theta_1^3)\cap \Theta_2^2$. Let
$y_1=q_2(y)\in \Theta_1^3$ and $y_0=q_1(y_1)$ be a stable map
$f:C\to L\subset\PP^r$ for some line $L$. Then $q_1^{-1}(y_0)\cap
\Theta_1^2\cong \PP \Ext^1(\cO_L,\cO_L(-1))=\PP^{r-2}$. From the
proof of Lemma \ref{lem4},  we see that the normal space to
$\Theta^2_1$ at $y_1$ is $\Ext^1(\cO_L,\cO_L(-1))\oplus \CC$ where
the summand $\CC$ parameterizes the smoothing of $C$, i.e. it comes
from the normal direction of $\Theta^1\cap\Theta^2$ in $\Theta^1$.
Since $y_1\in q_1^{-1}(y_0)\cap \Theta_1^2=\PP
\Ext^1(\cO_L,\cO_L(-1))$, $y_1$ can be thought of as a line in
$\Ext^1(\cO_L,\cO_L(-1))$. Certainly, the isomorphism classes of
sheaves in the family $\cE_1$ don't change in the smoothing
direction and along the line of $y_1$. Hence the Kodaira-Spencer map
on the normal space to $\Theta^2_1$ factors through
$\Ext^1(\cO_L,\cO_L(-1))/\CC\cdot y_1.$ Now by an easy calculation,
we find that $E=\cE_2|_{\PP^r\times\{y\}}$ is the extension sheaf
\[ 0\to \cO_L(-1)\oplus \cO_L(-1)\to E\to \cO_L\to 0\]
determined by $y_1$ and the image of the line $y$ in
$\Ext^1(\cO_L,\cO_L(-1))$. Hence $E$ is stable if and only if
$y\in \PP \left(\Ext^1(\cO_L,\cO_L(-1))\oplus \CC \right)$ does not
belong to the projective line spanned by $y_1$ and the smoothing
direction $\CC$, in which case $E=F\oplus \cO_L(-1)$ for a stable
sheaf $F$ with Hilbert polynomial $2m+1$. The projective line
spanned by $y_1$ and $\CC$ is precisely $\Theta^3_2\cap
q_2^{-1}(y_1)$.
\end{proof}

\begin{exam}\emph{ In fact, we can calculate all the stable
sheaves in $\cE_2$ by local calculations. 
Let $C$ be the
union of two curves $\{(t,s)\,|\,  ts=0\}$ and consider the family
of stable maps locally given by
\[
x_1=t^2,\quad x_2=s,\quad x_3=at,\quad a=ts.
\]
When $a\ne 0$, it is a family of smooth cubics and when $a=0$, we
have a stable map in $\Theta^2$. Then on the line of $a\in \CC$,
$$\cE_0=\cE_1=\CC[t,s]=\CC[t,s,a]/(ts-a)$$ and $\cE_2$ is the kernel of
\[
\CC[t,s]\lra \CC[t,s]/(ts)\lra t\CC[t^2]
\]
The central fiber of $\cE_2$ is then
\[
\cE_2/a\cE_2=\CC[t^2,s]/(t^2s^2)\cong
\CC[x_1,\cdots,x_r]/(x_1x_2^2,x_3^2,x_2x_3^2,x_3-x_1x_2,x_4,\cdots,x_r).
\]
This is a Cohen-Macaulay curve and hence stable. Of course, this
is a curve with two components, one of which is a reduced line and
the other is the double line thickened in a quadric surface in a
$\PP^3$. }\end{exam}

\begin{rema}\emph{
We will see in the subsequent section that $\Gamma^2$ is a
$\PP^2_{(1,2,2)}$-bundle over a $\PP^{r-1}\times \PP^{r-1}$-bundle
over $\PP^r$ by GIT, where $\PP^2_{(1,2,2)}$ is the weighted
projective space with weights $(1,2,2)$. Of course, the
$\PP^{r-1}\times \PP^{r-1}$-bundle over $\PP^r$ above parameterizes
pairs of intersecting lines in $\PP^r$ and $\PP^2_{(1,2,2)}$
parameterizes double coverings of a line. The proof of Lemma
\ref{lem5} tells us that the family $\cE_2$ of stable sheaves on
$\Gamma^2_2-\Gamma^1_2\cup\Gamma^3_2$ remains constant on the fibers
$\PP^2_{(1,2,2)}$ but depends only on the fiber
$\PP\Ext^1(\cO_{C'},\cO_{L_2}(-1))$ of $\Gamma^2_2\to \Gamma^2_1$.
}\end{rema}

Let $q_3:Q_3\to Q_2$ be the blow-up along $\Theta_2^3$ and let
$p_3:\bM_3\to \bM_2$ be its quotient by $SL(2)$. Let $\Theta_3^3$ be
the exceptional divisor and $\Theta^1_3, \Theta_3^2$ be the proper
transforms of $\Theta^1_2,\Theta^2_2$ respectively. Let $\Gamma_3^i$
be the quotient of $\Theta_3^i$ by $SL(2)$. Let $\cE_3'$ be the
pull-back of $\cE_2$ to $\PP^r\times Q_3$. As we saw above, if $y\in
\Theta_3^3$, $\cE_3'|_{\PP^r\times \{y\}}\cong F\oplus \cO_L(-1)$
for some line $L$ and a stable sheaf $F$ whose Hilbert polynomial is
$2m+1$. Obviously, $F$ is the destabilizing sheaf and by the
existence of relative Quot scheme \cite[Chapter 2]{HL} again, we
obtain a quotient homomorphism
\[
\cE_3'|_{\PP^r\times \Theta_3^3}\twoheadrightarrow A_3\] such that
for $y\in \Theta_3^3$, $A_3|_{\PP^r\times\{y\}}\cong \cO_L(-1)$
for some line $L$ depending on $y$. Let $\cE_3$ be the kernel of
the epimorphism
\[
\cE_3'\twoheadrightarrow \cE_3'|_{\PP^r\times
\Theta_3^3}\twoheadrightarrow A_3 .
\]

\begin{lemm}\lab{lem6} $\cE_3$ is a family of stable sheaves on $\PP^r$ and
hence we obtain an invariant morphism $\Psi_3:Q_3\to \bS$ which
induces a morphism
$$\Phi_3:\bM_3\to \bS.$$ \end{lemm}
\begin{proof} Let $y\in \Theta_3^3$, $y_2=q_3(y)$,
$y_1=q_2(y_2)$. From the proofs of Lemmas \ref{lem4} and \ref{lem5},
we find that the normal space to $\Theta_2^3$ in $Q_2$ at $y_2$ is
$$N_{\Theta_2^3/Q_2,y_2}\cong
\Ext^1(\cO_L,\cO_L(-1))/\CC\cdot y_1\oplus \CC \cong\CC^{r-1}$$
where the summand $\Ext^1(\cO_L,\cO_L(-1))/\CC\cdot y_1$ is the
normal space of $\Theta^3_2$ in $\Theta^1_2$ while the direct
summand $\CC$ is the normal direction to the divisor $\Theta^1_2$.
The family $\cE_2$ restricted to $\Theta_2^3$ is always of the
splitting form $F\oplus \cO_L(-1)$ where $F$ is the extension of
$\cO_L$ by $\cO_L(-1)$, i.e. the double line in a plane containing
$L$. Hence the Kodaira-Spencer map
\[
T_{y_2}Q_2\lra \Ext^1_{\PP^r}(F\oplus \cO_L(-1),F\oplus \cO_L(-1))
\]
induces a map
\[
N_{\Theta_2^3/Q_2,y_2}\lra \Ext^1_{\PP^r}(F,\cO_L(-1)).
\]
We claim this is injective and hence after elementary modification
$\cE_3|_{\PP^r\times\{y\}}$ becomes a nontrivial extension of $F$ by
$\cO_L(-1)$ which is obviously stable.

From the exactness of $0\to \cO_L(-1)\to F\to\cO_L\to 0$, we obtain
an exact sequence
\begin{equation}\label{eq-c1}
0\to\Ext^1(\cO_L,\cO_L(-1))/\CC\cdot y_1 \to \Ext^1(F,\cO_L(-1))\to
\Ext^1(\cO_L(-1),\cO_L(-1))\to 0 .
\end{equation}
Suppose $y\in \PP \Ext^1(\cO_L,\cO_L(-1))/\CC\cdot y_1$ and fix
$\tilde y\in \Ext^1(\cO_L,\cO_L(-1))$ representing $y$. Then from
the proofs of Lemmas \ref{lem4} and \ref{lem5}, we find that the
direction of $y$ in $N_{\Theta_2^3/Q_2,y_2}$ is given by the family
of extensions
\[
0\lra \cO_L(-1)\oplus \cO_L(-1)\lra E\lra \cO_L\lra 0
\]
whose extension class is $(y_1,\epsilon\tilde y)$. By direct
calculation, the elementary modification at $\epsilon=0$ gives us
the nontrivial extension with extension class $(y_1,\tilde y)$.
This is just the thickening of $L$ in a $\PP^3$ determined by
$y_1$ and $y$. Therefore, we find that the Kodaira-Spencer map
sends the summand $\Ext^1(\cO_L,\cO_L(-1))/\CC\cdot y_1$ of
$N_{\Theta_2^3/Q_2,y_2}$ isomorphically onto the same subspace
$\Ext^1(\cO_L,\cO_L(-1))/\CC\cdot y_1$ of $\Ext^1(F,\cO_L(-1))$.

For the other summand $\CC$ of $N_{\Theta_2^3/Q_2,y_2}$, consider
the locus $\Lambda$ of stable maps in $Q_0$ whose images are
planar. Then by direct local calculation, it is easy to see that
this locus is of codimension $r-2$ and its proper transform
$\Lambda_2$ in $Q_2$ intersects with $\Theta^1_2$ transversely
along $\Theta^3_2$. We choose an analytic curve $g:D\to\Lambda_2$
from a small disk $D$ in $\CC$ to $\Lambda_2$, which passes
through $y_2$ and moves away from $\Theta^1_2$. Let $\cE_g$ be the
pullback of $\cE_2$ to $\PP^r\times D$ via $(\id_{\PP^r},g)$. For
$t\ne 0$, the image of the stable map parameterized by
$q_1(q_2(g(t)))$ is a singular cubic plane curve $C'_t$ and
$\cE_2$ at $t$ is an extension of a skyscraper sheaf $\cO_{p_t}$
by the structure sheaf $\cO_{C_t'}$ of $C_t'$ for some $p_t\in
C_t'$. Without loss of generality, by applying linear
transformations, we may assume that $C'_t$ are all contained in a
fixed $\PP^2$ in $\PP^r$ for $t\ne 0$. The restriction of $\cE_g$
to $\Spec\CC[t]/(t^2)$ is an extension
\[
0\lra t\cdot (F\oplus \cO_L(-1))\lra \cE_g\lra F\oplus \cO_L(-1)\lra
0
\]
by flatness, where $F$ is the double of $L$ in $\PP^2$. By the
stability of $\cE_g$ for $t\ne 0$, the extension class in
$\Ext^1(F\oplus \cO_L(-1),F\oplus \cO_L(-1))$ has nontrivial
component $c\in\Ext^1(F,\cO_L(-1))$. Furthermore, since $\cE_g$ for
$t\ne 0$ is a family of stable sheaves supported in the plane
$\PP^2$ determined by $y_1$, if the image of $c$ in
$\Ext^1(\cO_L(-1),\cO_L(-1))$ by \eqref{eq-c1} were zero, then $c$
would come from $y_1\in \Ext^1(\cO_L,\cO_L(-1))$ and hence $c$ would
be trivial. Hence every sheaf in the family $\cE_3$ is stable as
desired.
\end{proof}

There are three possibilities for $y\in \Theta^3_3$ and in each case
one can calculate the sheaf $\cE_3$ at the point as in the example
below:\begin{enumerate}
\item If $y\in \PP \Ext^1(\cO_L,\cO_L(-1))/\CC$, then
$\cE_3|_{\PP^r\times\{y\}}$ is the thickening of $L$ in a $\PP^3$.
\item If $y$ is the normal direction to $\Theta^1_2$, then $\cE_3$
at $y$ is a nontrivial extension of a skyscraper sheaf $\CC_p$ for
some point $p\in L$ by the triple thickening of $L$ in a $\PP^2$
which contains $L$. \item If $y$ is not in either of the direct
summands, then $\cE_3$ at $y$ is a triple thickening of $L$ in a
quadric cone in a $\PP^3$ containing $L$.
\end{enumerate}
As observed in Remark \ref{rem4.8} and the above proof, the family
of stable sheaves parameterized by $\Theta^3_3$ is independent of
the factor $(\PP^7)^s$, i.e. we may choose any stable map
$y_0=q_1(y_1)$ to $L\cong\PP^1$ in calculating the stable sheaves in
$\cE_3$.

\begin{exam}\emph{ Case (2):
Consider the family of maps $f_a:\PP^1\to \PP^2$, locally given by
\[
\CC[x,y,a]\lra \CC[t,a]=E
\]
where $x\mapsto a(\rho t+t^2)$, $y\mapsto t^3$, $a\mapsto a$ for
$\rho\in\CC$. Before applying elementary modifications, we have
\[
E|_{a=0}=\CC[t]=\CC[y]\oplus t\CC[y]\oplus t^2\CC[y].
\]
The modification on $Q_1$ is taking the kernel $E_1$ of
\[
\CC[t,a]\lra \CC[t,a]/(a)\cong \CC[t]\lra t\CC[y]\oplus t^2\CC[y].
\]
Hence $E_1$ consists of elements of the form $f(t^3)+a\cdot
g(t,a)$ for polynomials $f,g$. To obtain $E_1|_{a=0}$, we have to
take the quotient by the submodule generated by $a$. It is easy to
see that the quotient by $(a)$ is
\[ \CC[y]\oplus at\CC[y]\oplus at^2\CC[y]\cong \CC[y]\oplus
at\CC[y]\oplus x\CC[y].
\]
The modification on $Q_3$ is taking the kernel $E_3$ of
\[
E_1\to E_1|_{a=0}\cong \CC[y]\oplus at\CC[y]\oplus x\CC[y]\lra
at\CC[y].
\]
If we quotient out $E_3$ by the submodule generated by $a$, we
obtain $$\CC[y]\oplus x\CC[y]\oplus x^2\CC[y]+\CC\cdot
a^2t=\CC[x,y]/(x^3)+\CC_{(0,\rho^3)}$$ which is a nontrivial
extension of $\CC_p$ where $p=(0,\rho^3)$ by $\CC[x,y]/(x^3)$, the
triple thickening of the line $x=0$ in the $xy$-plane, because
$x^2=a^2t(y-\rho^3)$ in $E_3/aE_3$.}\end{exam}

\begin{exam}
\emph{Case (3): Consider the family of maps $f_a:\PP^1\to \PP^3$,
locally given by
\[
\CC[x,y,z,a]\lra \CC[t,a]=E
\]
where $x\mapsto at^2$, $y\mapsto a^2t$, $z\mapsto t^3$, $a\mapsto
a$. As above, $E|_{a=0}=\CC[t]=\CC[z]\oplus t\CC[z]\oplus
t^2\CC[z].$ The calculations as above show that
$$\cE_3/a\cE_3=\CC[z]\oplus x\CC[z]\oplus y\CC[z]$$
and that $xy=0$, $y^2=0$, $yz-x^2=0$. This is the thickening of $L$
in the quadric cone $yz=x^2$.}\end{exam}

\bigskip

\subsection{Blow-downs}

In this subsection, we show that $\Phi_3:\bM_3\to \bS$ factors
through three blow-downs
\begin{equation}\lab{eq4.9}
\xymatrix{\bM_3\ar[r]^{\pi_4} &\bM_4\ar[r]^{\pi_5}
&\bM_5\ar[r]^{\pi_6} &\bM_6}
\end{equation}
where $\pi_4$ is a weighted blow-down along $\Gamma^2_3$, $\pi_5$
is a weighted blow-down along the proper transform of $\Gamma_3^3$
and $\pi_6$ is the smooth blow-down along the proper transform of
$\Gamma_3^1$. We will study the local geometry of the divisors and
show that the divisor to be contracted at each stage is a
(weighted) projective bundle and there exists a blow-down map
contracting the projective fibers. Then it is easy to see that the
morphism $\Phi_3$ remains constant on each contracted fiber and
hence it factors through $\Phi_6:\bM_6\to \bS$. Finally one can
directly check that the induced morphism
$$\Phi_6:\bM_6\lra \bS$$ is bijective
and hence we obtain an isomorphism $\bM_6\cong \bS$.

\medskip

Let us start with $\Gamma^1$, which is isomorphic to $\cM_0(\PP
\cU)$ over $Gr(2,r+1)$. Recall that
\[
\Gamma^1=\Theta^1/SL(2)
\]
where $\Theta^1$ is given by \eqref{eq916-1}. Since all the
transformations will take place over $Gr(2,r+1)$, we fix a line
$L\cong \PP^1\subset \PP^r$ to simplify the notation. In
particular, $\Theta^1$ is a $\PP (\Sym^3(\CC^2)\otimes
\CC^2)^s$-bundle on $Gr(2,r+1)$. By Corollary \ref{cor-b3}, the
normal bundle of $\Theta^1$, restricted to a fiber of $\Theta^1\to
Gr(2,r+1)$ over $L$, is
$$\cO_{\PP^7}(-1)^{2r-2}=\Ext^1(\cO_L,\cO_L(-1))^{\oplus 2}\otimes
\cO(-1)$$ and hence an analytic neighborhood $U^1$ of $\Gamma^1$ in
$\bM$ is equivalent to a bundle over $Gr(2,r+1)$ with fiber
$$\tilde U^1:=\cO_{\PP^7}(-1)^{2r-2}/\!/SL(2).$$
Therefore, a fiber of $\Gamma_1^1\to Gr(2,r+1)$ over $L$ is
\[
\tilde \Gamma_1^1:=\left(\PP (\Sym^3(\CC^2)\otimes \CC^2)\times
\PP(\CC^{r-1}\otimes \CC^2)\right)/\!/_{\cO(1,\lambda)}SL(2)
\]
which is the GIT quotient of $\PP^7\times \PP^{2r-3}$ with
linearization $\cO(1,\lambda)=\cO_{\PP^7}(1)\otimes
\cO_{\PP^{2r-3}}(\lambda))$ for $0<\lambda<1$. (Note that for
$\lambda<1$, the stable set is $(\PP^7)^s\times \PP^{2r-3}$.) Here
$SL(2)$ acts trivially on $\CC^{r-1}$ and as standard matrix
multiplication on $\CC^2$ for $\PP(\CC^{r-1}\otimes \CC^2)$.
Furthermore, an analytic neighborhood $U^1_1$ of $\Gamma^1_1$ is the
quotient of the line bundle
\[
\tilde
U^1_1:=\cO_{\PP^7\times\PP^{2r-3}}(-1,-1)/\!/_{\cO(1,\lambda)}SL(2)
\]
since the blow-up of $\cO_{\PP^7}(-1)^{2r-2}$ along the zero section
is $\cO_{\PP^7\times\PP^{2r-3}}(-1,-1)$. Here the linearization of
$\cO_{\PP^7\times\PP^{2r-3}}(-1,-1)$ comes from the compactification
$$\PP (\cO_{\PP^7\times\PP^{2r-3}}(-1,-1)\oplus \cO)$$ with
linearization
\[ \pi^*\cO_{\PP^7\times\PP^{2r-3}}(1,\lambda)\otimes \cO_{\PP (\cO(-1,-1)\oplus \cO
)}(-\epsilon Z)
\]
where $Z\cong \PP^7\times\PP^{2r-3}$ is the zero section
$\PP(0\oplus \cO)$ of
$$\pi:\PP(\cO_{\PP^7\times\PP^{2r-3}}(-1,-1)\oplus \cO)\to
\PP^7\times\PP^{2r-3}$$ and $\epsilon$ is a sufficiently small
positive number.

When $\lambda>>0$, the quotient
\[
\tilde
U^1_5:=\cO_{\PP^7\times\PP^{2r-3}}(-1,-1)/\!/_{\cO(1,\lambda)}SL(2)
\]
is the blow-up of $$\tilde
U^1_6:=\cO_{\PP^{2r-3}}(-1)^{8}/\!/SL(2)$$ along the zero section
$$\tilde \Gamma^1_6:= \PP^{2r-3}/\!/SL(2)\cong Gr(2,r-1)$$ and we have
a blow-down map $\tilde U^1_5\lra \tilde U^1_6$ to a neighborhood of
$\tilde\Gamma^1_6$. This Grassmannian $\tilde\Gamma^1_6$
parameterizes the choice of a $\PP^3$ containing the line $L$ and
each point will give us a thickening of $L$ in the chosen $\PP^3$.

We will see that the blow-up process in the previous section can be
described by variation of GIT quotients \cite{Thad, DH}, as we vary
$\lambda$ from $1^-$ to $\infty$. It is easy to see by the
Hilbert-Mumford criterion that the quotient varies only at
$\lambda=1$ and $\lambda=3$. Then one finds the $\CC^*$-fixed locus
and the weight space decomposition of the normal bundle to each
fixed point component. It is just an elementary exercise to show
that the variations at $\lambda=1$ and $\lambda=3$ are flips, i.e. a
blow-up followed by a blow-down, as follows.
\begin{enumerate}
\item The flip at $\lambda=1$:
\[
\tilde\Gamma^1_1=\PP^7\times \PP^{2r-3}/\!/_{\cO(1,1^-)}SL(2)
\longleftarrow \tilde\Gamma_2^1 \longrightarrow \PP^7\times
\PP^{2r-3}/\!/_{\cO(1,2)}SL(2)
\]
is the composition of a blow-up and a blow-down. The first map is
the smooth blow-up along a $\PP^2_{(1,2,2)}$-bundle over
$\PP^1\times \PP^{r-2}$, where $\PP^2_{(1,2,2)}$ is the weighted
projective space with weights $(1,2,2)$. This blow-up center
coincides with the fiber of the intersection $\Gamma^1_1\cap
\Gamma^2_1$ over $Gr(2,r+1)$. The second map is the weighted
blow-up along a $\PP^{r-1}$-bundle over $\PP^1\times \PP^{r-2}$
with weights $(1,2,2)$ on each normal space.
\item The flip at $\lambda=3$:
\[
\PP^7\times \PP^{2r-3}/\!/_{\cO(1,2)}SL(2) \longleftarrow
\tilde\Gamma_4^1 \longrightarrow \PP^7\times
\PP^{2r-3}/\!/_{\cO(1,4)}SL(2)=:\tilde\Gamma_5^1
\]
is the composition of a blow-up and a blow-down. The first map is
the smooth blow-up along a $\PP^4_{(1,2,2,3,3)}$-bundle over
$\PP^1\times \PP^{r-2}$ where $\PP^4_{(1,2,2,3,3)}$ is the
weighted projective space with weights $(1,2,2,3,3)$. The second
map is the weighted blow-up along a $\PP^{r-3}$-bundle over
$\PP^1\times \PP^{r-2}$ with weights $(1,2,2,3,3)$ on each normal
space.
\end{enumerate}
When $\lambda>3$, $\tilde\Gamma^1_5=\PP^7\times
\PP^{2r-3}/\!/_{\cO(1,\lambda)}SL(2)$ is a $\PP^7$-bundle over
$$\PP(\CC^{r-1}\otimes\CC^2)/\!/SL(2)\cong Gr(2,r-1)=:\tilde
\Gamma^1_6.$$

If we let $\tilde \Gamma_3^1$ be the fiber product of $\tilde
\Gamma_2^1$ with $\tilde \Gamma_4^1$, we obtain the following
diagram.
\begin{equation}\lab{diag4.1}
\xymatrix{ &&\tilde\Gamma^1_3\ar[dl] \ar[dr]
\\
&\tilde\Gamma^1_2\ar[dl] \ar[dr] && \tilde\Gamma^1_4\ar[dl]
\ar[dr]
\\
\tilde\Gamma_1^1\ar[d]_{\PP^{2r-3}} &&\PP^7\times
\PP^{2r-3}/\!/_{\cO(1,2)}SL(2) && \tilde\Gamma_5^1\ar[d]^{\PP^7}\\
\cM_0(\PP^1,3)&&&& \tilde\Gamma_6^1 }
\end{equation}
The left and right vertical maps are projective bundles and all
other maps are blow-ups. The two maps from $\tilde\Gamma_3^1$ are
blow-ups because the blow-up centers in $\PP^7\times
\PP^{2r-3}/\!/_{\cO(1,2)}SL(2) $ are transversal. By our
construction, it is easy to see that $\Gamma^1_i$ is exactly a
$\tilde\Gamma_i^1$-bundle over $Gr(2,r+1)$ for $i\le 3$.

Similarly, we can study the variation of the GIT quotient
$$\tilde U^1_1=\cO_{\PP^7\times\PP^{2r-3}}(-1,-1)/\!/SL(2)$$ with
linearization
\[ \pi^*\cO_{\PP^7\times\PP^{2r-3}}(1,\lambda)\otimes \cO_{\cO(-1,1)}(-\epsilon Z)
\]
where $Z\cong \PP^7\times\PP^{2r-3}$ is the zero section as we vary
$\lambda$ from $1^-$ to $\infty$. As above, the GIT quotient varies
only at $\lambda=1$ and $\lambda=3$.
\begin{enumerate}\item
The wall crossing at $\lambda=1$ takes place over the $\CC^*$-fixed
point component $B_1$, which is the restriction of the line bundle
$\cO(-1,-1)$ to the flip base $\PP^1\times\PP^{r-2}$ for $\tilde
\Gamma^1$ at $\lambda=1$ above. The flip at $\lambda=1$ is the
composition
\[
\tilde U^1_1=\cO_{\PP^7\times
\PP^{2r-3}}(-1,-1)/\!/_{\lambda=1^-}SL(2) \longleftarrow \tilde
U_2^1 \longrightarrow \cO_{\PP^7\times
\PP^{2r-3}}(-1,-1)/\!/_{\lambda=2}SL(2)
\]
of the smooth blow-up along a $\PP^2_{(1,2,2)}$-bundle over $B_1$
and the weighted blow-up along a $\PP^{r-1}$-bundle over $B_1$.
\item The wall crossing at $\lambda=3$ takes place over the
$\CC^*$-fixed component $B_2\cong \PP^1\times \PP^{r-2}$, which is
the zero section of the restriction of the line bundle $\cO(-1,-1)$
to the flip base $\PP^1\times \PP^{r-2}$ for $\tilde\Gamma^1$ at
$\lambda=3$ above. The flip at $\lambda=3$ is the composition
\[
\cO_{\PP^7\times \PP^{2r-3}}(-1,-1)/\!/_{\lambda=2}SL(2)
\longleftarrow \tilde U_4^1 \longrightarrow \cO_{\PP^7\times
\PP^{2r-3}}(-1,-1)/\!/_{\lambda=4}SL(2)=\tilde U_5^1
\]
is the composition of the smooth blow-up along a
$\PP^4_{(1,2,2,3,3)}$-bundle over $B_2$ and the weighted blow-up
along a $\PP^{r-2}$-bundle over $B_2$ with weights $(1,2,2,3,3)$ on
each normal space.
\end{enumerate}
In summary, we have the following diagram of blow-ups for $\tilde
U^1_j$:
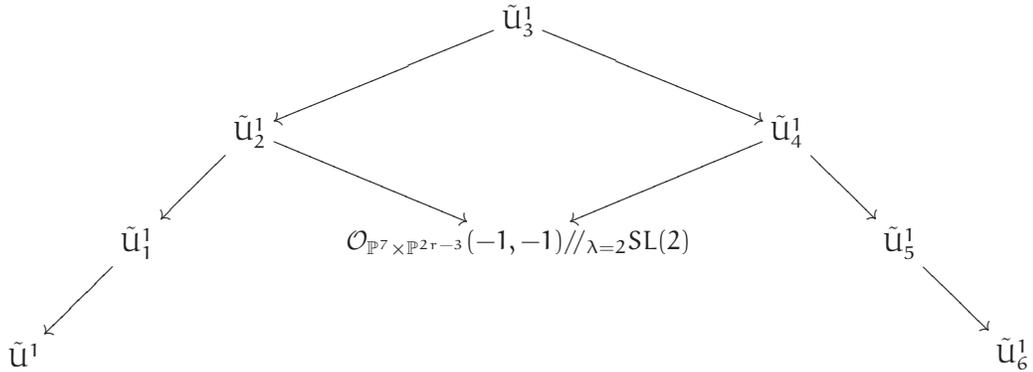
\begin{equation}\lab{diag4-c1}
\xymatrix{ &&&\tilde U^1_3\ar[dl] \ar[dr]
\\
&&\tilde U^1_2\ar[dl]  \ar[dr] && \tilde U^1_4\ar[dl] \ar[dr]
\\
&\tilde U_1^1\ar[dl]  &&\cO_{\PP^7\times
\PP^{2r-3}}(-1,-1)/\!/_{\lambda=2}SL(2) && \tilde U_5^1\ar[dr] \\
\tilde U^1 &&&&&& \tilde U_6^1 }
\end{equation}
Here $\tilde U^1_3$ is the fiber product of $\tilde U^1_2$ and
$\tilde U^1_4$ over $\cO_{\PP^7\times
\PP^{2r-3}}(-1,-1)/\!/_{\lambda=2}SL(2)$. As the blow-up centers in
$\cO_{\PP^7\times \PP^{2r-3}}(-1,-1)/\!/_{\lambda=2}SL(2)$ are
transverse, the two maps from $\tilde U^1_3$ to $\tilde U^1_2$ and
$\tilde U^1_4$ are blow-ups.

Let $U^1_j\subset\bM_j$, $j=1,2,3$ be the inverse image of the
neighborhood $U^1$ of $\Gamma^1$. Here $U^1$ is a bundle over
$Gr(2,r+1)$ with fiber $\tilde U^1$.  From our construction, it is
easy to check that the restrictions of the blow-ups
$$\bM_3\lra \bM_2\lra\bM_1\lra\bM$$ to $U^1$ coincide with the
three blow-ups on the left side of \eqref{diag4-c1}. If we can
check that the first blow-down $\tilde U^1_3\to \tilde U^1_4$ is
extended to a global contraction map $p_4:\bM_3\to \bM_4$ along
$\Gamma^2_3$, then simply by gluing with
$\bM_4-\Gamma^1_4\cup\Gamma^3_4$ where
$\Gamma^i_4=p_4(\Gamma^i_3)$, the blow-downs
$$\tilde U^1_4\lra \tilde U^1_5\lra \tilde U^1_6$$ extends to a
global maps
\[ \bM_4\lra \bM_5\lra\bM_6\]
which contracts along the proper transform $\Gamma^3_4$ of
$\Gamma^3_3$ and then along the proper transform $\Gamma^1_5$ of
$\Gamma^1_3$, because $\Gamma^1_j$ and $\Gamma^3_j$ are both
contained in $U^1_j$.

\medskip

So let us now study the blow-ups of $\Gamma^2$. Recall that
$$\Gamma^2=\Theta^2/SL(2)$$ where $\Theta^2$ is a $\PP^{r-1}$-bundle
over a bundle over $Gr(2,r+1)$ with fiber
\[
[\PP^1\times \PP(\Sym^2(\CC^2)\otimes \CC^2)]^s.
\]
As before, we fix $L\in Gr(2,r+1)$ to simplify our notation because
everything takes place over $Gr(2,r+1)$ in a neighborhood of
$\Theta^2$. Consider the variation of the GIT quotient
\[
\PP^1\times\PP(\Sym^2(\CC^2)\otimes \CC^2)\git_{\cO(1,\lambda)}
SL(2)
\]
as we vary $\lambda$ from $1$ to $0^+$. It is again an elementary
exercise of GIT variation (\cite{Thad,DH}) to see that the quotient
varies only at $\lambda=1/2$ but there are no stable points if
$\lambda<1/2$. The variation at $\lambda=1/2$ tells us that
\[
\PP^1\times\PP(\Sym^2(\CC^2)\otimes \CC^2)\git_{\cO(1,1)} SL(2)
\]
is a $\PP^2_{(1,2,2)}$-bundle over $\PP^1=L$. Since the
$\PP^1$-bundle over $Gr(2,r+1)$ is a $\PP^{r-1}$-bundle over
$\PP^r$, $\Gamma^2$ is a $\PP^2_{(1,2,2)}$-bundle over a
$(\PP^{r-1}\times \PP^{r-1})$-bundle over $\PP^r$. Obviously, the
$(\PP^{r-1}\times \PP^{r-1})$-bundle over $\PP^r$ parameterizes
(ordered) pairs of lines meeting at a point while $\PP^2_{(1,2,2)}$
parameterizes double coverings for the first (or second) line.

We can easily keep track of $\Gamma^2$ through the blow-ups
$p_1,p_2,p_3$. Firstly, $\Gamma_1^2$ is just the blow-up of
$(\PP^{r-1}\times \PP^{r-1})$ along the diagonal. Secondly,
$\Gamma_2^2$ is a $\PP^{r-1}$-bundle over $\Gamma_1^2$. From the
proof of Lemma \ref{lem5}, we see that the normal bundle to
$\Theta^2_1$ is independent of the fiber $\PP^2_{(1,2,2)}$ and
hence $\Gamma_2^2$ is in fact a $\PP^2_{(1,2,2)}\times
\PP^{r-1}$-bundle over a $bl_{\PP^{r-1}}(\PP^{r-1}\times
\PP^{r-1})$-bundle over $\PP^r$. Finally $\Gamma^2_3$ replaces
$\PP^2_{(1,2,2)}\times \PP^{r-1}$ by $\PP^2_{(1,2,2)}\times
bl_{\PP^1}\PP^{r-1}$. We claim that the $\PP^2_{(1,2,2)}$ fibers
of $\Gamma^2_3$ can be contracted to give us a map $\bM_3\to
\bM_4$. Indeed, this claim easily follows from the following two
observations. Firstly, from our description of the map $\tilde
U^1_3\to \tilde U^1_4$, a neighborhood of $U^1_3\cap\Gamma^2_3$ is
a $\cO_{\PP^2_{(1,2,2)}}(-1)$-bundle over $U^1_3\cap\Gamma^2_3$
and hence can be contracted to a $\CC^3$ bundle. Secondly, any
pair of distinct intersecting lines in $\PP^r$ can be sent to
another pair of distinct intersecting lines by the action of
$GL(r+1)$ and hence an analytic neighborhood of any point in
$\Gamma^2_3$ is isomorphic to an analytic neighborhood for a point
in $U^1_3\cap\Gamma^2_3$. Since the action of an element in
$GL(r+1)$ preserves the blow-down $U^1_3\to U^1_4$, we obtain the
desired map $p_4:\bM_3\to\bM_4$ which is a weighted blow-up along
$\Gamma^2_4$. Therefore, we have analytic maps
$\bM_3\to\bM_4\to\bM_5\to\bM_6$ contracting the divisors
$\Gamma^2_3, \Gamma^3_4,\Gamma^1_5$ respectively.

Finally from our construction of the stable sheaves in the previous
subsection, it is an easy but tedious exercise to check that
set-theoretically the map $\Phi_3:\bM_3\to \bS$ factors through the
maps $\bM_3\to\bM_4\to\bM_5\to\bM_6$ and hence we obtain an analytic
map
$$\Phi_6:\bM_6\lra \bS.$$ Furthermore, it is easy to see that
$\Phi_6$ is bijective. Therefore, we obtain the following by
Zariski's main theorem.
\begin{theo}\lab{mainthm}
$\bS$ is obtained from $\bM$ by blowing up along $\Gamma^1$,
$\Gamma^2_1$, $\Gamma^3_2$ and then blowing down along $\Gamma^2_3$,
$\Gamma^3_4$, $\Gamma^1_5$.
\end{theo}
The following diagram summarizes the results of this paper.
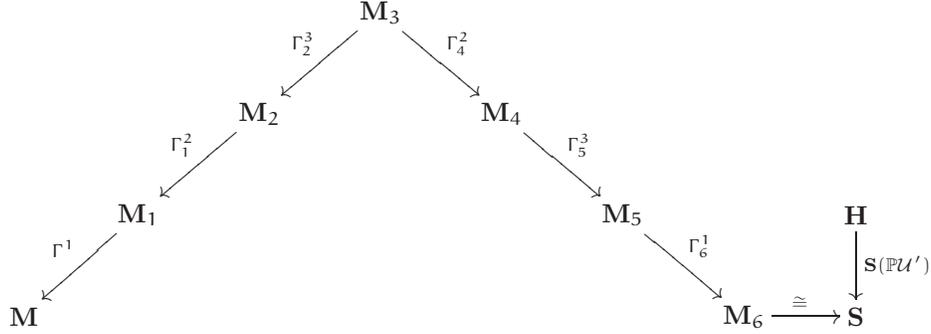
\[
\xymatrix{ &&&\bM_3\ar[dl]_{\Gamma_2^3}\ar[dr]^{\Gamma_4^2}\\
&&\bM_2\ar[dl]_{\Gamma_1^2}&&\bM_4\ar[dr]^{\Gamma^3_5}\\
&\bM_1\ar[dl]_{\Gamma^1}&&&&\bM_5\ar[dr]^{\Gamma^1_6}&&\bH\ar[d]^{\bS(\PP\cU')}\\
\bM&&&&&&\bM_6\ar[r]^\cong & \bS }\] All the arrows are blow-ups and
the blow-up centers are indicated above the arrows.

\bigskip

\section{Cohomology calculation}
We use Theorem \ref{mainthm} to calculate the Betti numbers of
$\bS$, by the blow-up formula \cite{GH}. We define the Poincar\'e
polynomial of a topological space $X$ as
\[
P_t(X)=\sum_it^i\dim H^i(X).
\]
From \cite{KM}, the Poincar\'e polynomial of $\bM=\cM_0(\PP^r,3)$ is
\[
P_t(\bM)=\left(\frac{1-t^{2r+10}}{1-t^6}+2\frac{t^4-t^{2r+4}}{1-t^4}\right)
\frac{(1-t^{2r+2})}{(1-t^2)}\frac{(1-t^{2r+2})(1-t^{2r})}{(1-t^2)(1-t^4)}.
\]
The first blow-up adds
\[
P_t(\Gamma^1)\frac{t^2-t^{4r-4}}{1-t^2}=(1+t^2+2t^4+t^6+t^8)
\frac{1-t^{2r+2}}{1-t^2}\frac{1-t^{2r}}{1-t^4}\frac{t^2-t^{4r-4}}{1-t^2}.
\]
The second blow-up adds
\[
P_t(\Gamma^2_1)\frac{t^2-t^{2r}}{1-t^2}=(1+t^2+t^4)
\frac{1-t^{2r}}{1-t^2}\left(
\frac{1-t^{2r}}{1-t^2}+\frac{t^2-t^{2r-2}}{1-t^2}
 \right) \frac{1-t^{2r+2}}{1-t^2}\frac{t^2-t^{2r}}{1-t^2}.
\]
The third blow-up adds
\[
P_t(\Gamma^3_2)\frac{t^2-t^{2r-2}}{1-t^2}=\left(
(1+t^2)(1+t^2+2t^4+t^6+t^8)+t^2(1+t^2)(1+t^2+t^4)\right)\]
\[
\cdot\frac{1-t^{2r-2}}{1-t^2}\frac{1-t^{2r+2}}{1-t^2}\frac{1-t^{2r}}{1-t^4}\frac{t^2-t^{2r-2}}{1-t^2}.
\]
The first blow-down subtracts
\[
P_t(\Gamma^2_4)\cdot (t^2+t^4)=\left[\frac{1-t^{2r}}{1-t^2}
\left(\frac{1-t^{2r}}{1-t^2}+\frac{t^2-t^{2r-2}}{1-t^2}
 \right)+(1+t^2)\frac{1-t^{2r-2}}{1-t^2}\frac{t^2-t^{2r-2}}{1-t^2}\right]\]
\[ \cdot \frac{1-t^{2r}}{1-t^2}\frac{1-t^{2r+2}}{1-t^2} (t^2+t^4).
\]
The second blow-down subtracts
\[
P_t(\Gamma^3_5)\frac{t^2-t^{10}}{1-t^2}=
(1+t^2)\left(\frac{1-t^{2r-2}}{1-t^2}\right)^2
\frac{1-t^{2r+2}}{1-t^2}\frac{1-t^{2r}}{1-t^4}\frac{t^2-t^{10}}{1-t^2}.
\]
The third blow-down subtracts
\[
P_t(\Gamma^1_6)\frac{t^2-t^{16}}{1-t^2}=
\frac{1-t^{2r-2}}{1-t^2}\frac{1-t^{2r-4}}{1-t^4}
\frac{1-t^{2r+2}}{1-t^2}\frac{1-t^{2r}}{1-t^4}\frac{t^2-t^{16}}{1-t^2}.
\]

In summary, we obtain
\begin{theo}
The Poincar\'e polynomial of $\bS$ is
\[
\left(\frac{1-t^{2r+10}}{1-t^6}+2\frac{t^4-t^{2r+4}}{1-t^4}\right)
\frac{(1-t^{2r+2})}{(1-t^2)}\frac{(1-t^{2r+2})(1-t^{2r})}{(1-t^2)(1-t^4)}
\]
\[ + (1+t^2+2t^4+t^6+t^8)
\frac{1-t^{2r+2}}{1-t^2}\frac{1-t^{2r}}{1-t^4}\frac{t^2-t^{4r-4}}{1-t^2}
\]
\[
+ (1+t^2+t^4) \frac{1-t^{2r}}{1-t^2}\left(
\frac{1-t^{2r}}{1-t^2}+\frac{t^2-t^{2r-2}}{1-t^2}
 \right) \frac{1-t^{2r+2}}{1-t^2}\frac{t^2-t^{2r}}{1-t^2}
\]
\[
+ \left( (1+t^2)(1+t^2+2t^4+t^6+t^8)+t^2(1+t^2)(1+t^2+t^4)\right)\]
\[
\cdot\frac{1-t^{2r-2}}{1-t^2}\frac{1-t^{2r+2}}{1-t^2}\frac{1-t^{2r}}{1-t^4}\frac{t^2-t^{2r-2}}{1-t^2}
\]
\[
- \left[\frac{1-t^{2r}}{1-t^2}
\left(\frac{1-t^{2r}}{1-t^2}+\frac{t^2-t^{2r-2}}{1-t^2}
 \right)+(1+t^2)\frac{1-t^{2r-2}}{1-t^2}\frac{t^2-t^{2r-2}}{1-t^2}\right]\]
\[ \cdot \frac{1-t^{2r}}{1-t^2}\frac{1-t^{2r+2}}{1-t^2} (t^2+t^4)
\]
\[
- (1+t^2)\left(\frac{1-t^{2r-2}}{1-t^2}\right)^2
\frac{1-t^{2r+2}}{1-t^2}\frac{1-t^{2r}}{1-t^4}\frac{t^2-t^{10}}{1-t^2}
\]
\[
-\frac{1-t^{2r-2}}{1-t^2}\frac{1-t^{2r-4}}{1-t^4}
\frac{1-t^{2r+2}}{1-t^2}\frac{1-t^{2r}}{1-t^4}\frac{t^2-t^{16}}{1-t^2}.
\]
\end{theo}
In particular, when $r=3$, we obtain
\[
P_t(\bH)=P_t(\bS)=1+2t^2+6t^4+10t^6+16t^8+19t^{10}+22t^{12}+19t^{14}+16t^{16}+10t^{18}
+6t^{20}+2t^{22}+t^{24}
\]
which coincides with the calculation in \cite{EPS}.


\bibliographystyle{amsplain}

\end{document}